\documentclass[11pt, a4paper]{article}
\usepackage[english]{babel}
\usepackage{a4wide}
\usepackage{times}
\usepackage[centertags,reqno]{amsmath}
\usepackage{amssymb}
\usepackage{amsthm}
\usepackage{color}
\usepackage[title]{appendix}

\def\nc{\newcommand}

\newcounter{licznik}[section]

\theoremstyle{plain}
\newtheorem{thm}[licznik]{THEOREM}

\newtheorem{lemma}[licznik]{LEMMA}

\theoremstyle{definition}
\newtheorem{definition}[licznik]{DEFINITION}
\theoremstyle{remark}
\newtheorem{remarkx}[licznik]{REMARK}
\newenvironment{remark}
  {\pushQED{\qed}\remarkx}
 {\popQED\endremarkx}

\newtheorem*{examplex}{EXAMPLE}
\newenvironment{example}
  {\pushQED{\qed}\examplex}
  {\popQED\endexamplex}

\makeatletter
\newcommand\assumptionlabel[1]{\hspace\labelsep
                               \normalfont\bfseries #1 \gdef\@currentlabel{#1}}
\newenvironment{assumption}
               {\list{}{\labelwidth\z@ \itemindent-\leftmargin
                        }}
               {\endlist}
\makeatother

\nc\com[1]{\textcolor{red}{#1}}
\nc\mop[1]{\hspace{#1pt}}
\nc{\laa}{\mathcal A}
\nc{\gie}{\mathcal G}
\nc{\lo}{\mathcal O}
\nc\ha{\mathcal{H}}
\nc\gk{\mathcal G}
\nc\ga{\mathcal A}
\nc{\prob}{\mathbb{P}}
\nc\pe{\mathcal{P}}
\nc{\lde}{\mathcal{D}}
\nc\eh{\mathcal{H}}
\nc{\ef}{\mathcal{F}}
\nc{\er}{\mathbb{R}}
\nc{\erp}{\mathbb{R}_+}
\nc{\en}{{\mathbb{N}}}
\nc{\es}{\mathcal{S}}
\nc{\te}{\mathcal{T}}
\nc{\borel}{\mathcal{B}}
\nc\bse{\mathcal{E}}
\nc\ve{\varepsilon}
\nc{\tl}{\tilde}
\nc{\ee}{{\mathbb{E}\,}}
\nc{\ind}[1]{1_{\{#1\}}\,}
\nc{\supp}{\mathop{\rm supp} \nolimits}
\nc{\conv}{\mathop{\rm conv} \nolimits}
\nc{\essinf}{\mathop{\rm ess\, inf} \limits}
\nc{\esssup}{\mathop{\rm ess\, sup} \limits}

\begin{document}


\title{Infinite horizon stopping problems with (nearly) total reward criteria\footnote{Part of the research of the first author was performed during a visit to the Hausdorff Research Institute for Mathematics at the University of Bonn in the framework of the Trimester Program Stochastic Dynamics in Economics and Finance. Research of both authors was supported in part by MNiSzW Grant no. UMO-2012/07/B/ST1/03298}}
\date{$\,$Date: 2014-01-27 15:20:45 $\,$}

\author{
Jan Palczewski\footnote{\normalfont\rm School of Mathematics, University of Leeds, Leeds LS2 9JT, UK (e-mail: J.Palczewski@leeds.ac.uk)}
\and
\L ukasz Stettner\footnote{Institute of Mathematics, Polish Academy of Sciences, Sniadeckich
8, 00-956 Warszawa, Poland and  Vistula University, Warszawa, Poland (e-mail: \mbox{stettner@impan.gov.pl}).}
}

\maketitle

\begin{abstract}
We study an infinite horizon optimal stopping problem which arises naturally in the optimal timing of a firm/project sale or in the valuation of natural resources: the functional to be maximised is a sum of a discounted running reward and a discounted final reward. The running and final rewards as well as the instantaneous interest rate (used for calculating discount factors) depend on a Feller-Markov process modelling the underlying randomness in the world, such as prices of stocks, prices of natural resources (gas, oil, etc.), or factors influencing the interest rate. For years, it had seemed sensible to assume that interest rates were uniformly separated from 0, which is needed for the existing theory to work. However, recent developments in Japan and in Europe showed that interest rates can get arbitrarily close to 0. In this paper we establish the feasibility of the stopping problem, prove the existence of optimal stopping times and a variational characterisation (in the viscosity sense) of the value function when interest rates are \textit{not} uniformly separated from 0. Our results rely on certain ergodic properties of the underlying (non-uniformly) ergodic Markov process. We provide several criteria for diffusions and jump-diffusions.
\end{abstract}

\section{Introduction}
The theory of optimal stopping has recently seen its renaissance due to applications in finance and operations research (e.g., pricing of American options, optimal timing of a sale or valuation of natural resources, see \cite{Pham2009} and references therein, or pricing of swing options with applications in energy trading \cite{Carmona2008}), and statistics (e.g., sequential hypothesis testing, see \cite{Peskir2006, shiryaev1978} and references therein). In some applications, such as the valuation of American options, there is a natural bound for stopping times -- the maturity date of the option. This results in finite horizon stopping problems. In others, like finding an optimal time to sell a stock/business or valuing natural resources, the horizon is infinite leading to an optimal stopping problem of the following form:
\begin{equation}\label{eqn:general}
w(x) = \sup_{\tau} \limsup_{T \to \infty} \ee^x \Big\{ \int_0^{\tau \wedge T} e^{-\int_0^s r(X_u) du} f(X_s) ds
 + e^{-\int_0^{\tau \wedge T}r(X_u) du} g(X_{\tau \wedge T}) \Big\}.
\end{equation}
Here, the stochastic process $X(t)$ models the underlying randomness in the world, such as the price of a stock, the price of a natural resource (gas, oil, etc.), or factors influencing the interest rate. The running profit/cost is represented by $f$ while the proceeds from the final sale at time $\tau$ or the closure of the production facility are given by $g$. The function $r$ corresponds to an instantaneous interest rate that is used for discounting of future cash-flows. For years, it had seemed sensible to assume that interest rates were uniformly separated from $0$. However, recent developments in Japan and in Europe showed that (even nominal) interest rates for government bonds can get arbitrarily close to $0$.

The uniform separation of the discount rate from $0$ (i.e., $\inf_x r(x) > 0$) is in line with a large strand of literature in infinite-horizon optimal stopping \cite{bensoussan1978, Peskir2006, Pham2009,  Robin1978}; commonly, the discount rate is a positive constant.  This ensures, under appropriate assumptions on the growth of $f$ and $g$, that the value function is finite and allows to approximate the infinite horizon stopping problem with finite horizon ones. An abolition of the discounting or a relaxation of its separation from zero brings in a lot of difficulties. In particular, the integrability in the functional is at risk; in Section \ref{sec:no_disc} we show that this is indeed the case. A solution to this problem is to impose restrictive assumptions on $f$ and $g$. In martingale approaches it is common to assume that $\ee^x \{ \int_0^\infty e^{-\int_0^s r(X_u) du} f^-(X_s) ds \} < \infty$ for every $x$ and that the family $\{ e^{-\int_0^{\tau}r(X_u) du} g^-(X_\tau):\ \text{$\tau$-a stopping time}\}$, is uniformly integrable, see \cite{Oksendal2005, Peskir2006} and references therein. Alternatively, one can take $f$ to be non-positive (then the integral term penalises for waiting only) \cite{Peskir2006, shiryaev1978}.


A general stopping problem without the restrictions on $f$ was studied in \cite{Robin1981,  Stettner1986, Stettner1989a} for uniformly geometrically ergodic Feller-Markov processes. Such processes converge exponentially fast to their invariant measure and the speed of this convergence is independent of the starting point. Examples are usually constrained to processes on compact state spaces. An attempt to generalise these results was made in \cite{Stettner1989}, where the author assumed that trajectories of the process can be split into excursions with square integrable lengths between two compact sets. In this setting optimality was studied within a narrow class of stopping rules.

The aim of this paper is to analyse the optimal stopping problem \eqref{eqn:general} for \textit{non-uniformly ergodic} Feller-Markov processes with minimal assumptions on $f$ and $g$, and a general discount rate $r$ which is only assumed to be non-negative. Specifically, assume that the state of the world is described by a right-continuous time homogeneous Markov process $\big(X(t)\big)$ defined on a locally compact separable space $E$ endowed with a metric $\rho$ with respect to which every closed ball is compact. The Borel $\sigma$-algebra on $E$ is denoted by $\bse$. Let $P_t$ be the semigroup generated by the process $(X_t)$, i.e., $P_t \phi(x) = \ee^x \{ \phi(X_t) \}$ for any bounded measurable function $\phi:E \to \er$. The transition measure is given by $P_t(x, A) = \prob^x \{ X_t \in A \}$ for $A \in \bse$.

We make the following standing assumptions:
\begin{assumption}
\item[(A1)] \label{ass:weak_feller}
(Weak Feller property)
$$
P_t\, \mathcal{C}_0 \subseteq \mathcal{C}_0,
$$
where $\mathcal C_0$ is the space of continuous bounded functions $E \to \er$ vanishing in infinity.
\item[(A2)] \label{ass:ergodic} $(X_t)$ is ergodic, i.e., there is a unique probability measure $\mu$ on $\bse$ such that for any $x \in E$
\[
\lim_{t \to \infty} \|P_t(x, \cdot) - \mu(\cdot) \|_{TV} = 0,
\]
where $\| \cdot \|_{TV}$ denotes the total variation norm.
\item[(A3)] Functions $f, g: E \to \er$ are continuous and bounded. \label{ass:continuity}
\end{assumption}
Assumption \ref{ass:weak_feller} is commonly imposed and, in fact, necessary, to establish the existence of optimal stopping times for general weak Feller processes (for the necessity see the example at the end of Section 3.1 in \cite{Palczewski2008}). The class of weakly Feller processes \ref{ass:weak_feller} comprises Levy processes \cite[Theorem 3.1.9]{Applebaum2004}, solutions to stochastic differential equations with continuous bounded coefficients driven by Levy processes (for details see \cite[Theorem 6.7.2]{Applebaum2004}) and many diffusions and jump-diffusions with unbounded coefficients encountered in modelling. Notice that a weak Feller process can be assumed to be a standard Markov process, see \cite[Theorem 3.14]{Dynkin1965}.

The ergodicity assumption \ref{ass:ergodic} will be used to remove the requirement for the discount rate to be bounded from below by a strictly positive constant. It provides information about the behaviour of the integral part of the functional \eqref{eqn:general}. Indeed, when the function $f$ is not restricted to either positive or negative values, the integral of $f$ with respect to the invariant measure $\mu$ determines the solution of the problem. For an undiscounted problem, i.e., $r \equiv 0$, it is optimal never to stop whenever $\mu(f) > 0$. On the other hand, the case of $\mu(f) < 0$ exhibits a much more interesting behaviour which is studied in detail in this paper.

The continuity in assumption \ref{ass:continuity} can be relaxed when the process is strongly Feller, for example, for locally non-degenerate diffusions or jump-diffusions.

The most important of our contributions towards the solution of the stopping problem \eqref{eqn:general} are distilled in the following theorem:
\begin{thm}\label{thm:intro_undis}
Assume that $\mu(f) < 0$, function $r$ is non-negative, continuous and bounded, and either
\begin{itemize}
\item[i)] there is $K: E \to (0, \infty)$ bounded on compacts and $h:[0, \infty) \to \er_+$ such that
\[
\| P_t(x, \cdot) - \mu(\cdot) \|_{TV} < K(x) h(t),
\]
and $\int_0^\infty h(t) dt < \infty$,
\item[ii)] $\ee^x \{ K(X_T) \} < \infty$ for any $x \in E$ and $T \ge 0$,
\item[iii)] the process $(X_t)$ is strongly Feller, or for any compact set $L \subset E$ there is $\alpha > 0$ such that $\sup_{x \in L,\ T \ge 0} \ee^x \{ K(X_T)^{1+\alpha} \} < \infty$,
\end{itemize}
or
\begin{itemize}
\item[i')] for any $\delta, \ve>0$ and a compact set $K \subset E$ there is $N>0$ and $p>0$ such that for $n\ge N$ and $x \in K$ we have
\[
\prob^x\Big\{\Big|\frac1{\int_0^{n \delta} e^{-\int_0^s r(X_u) du}ds} \int_0^{n \delta} e^{-\int_0^s r(X_u) du} \big(f(X_s) - \mu(f)\big) ds \Big|> \ve \Big\} \le e^{-p (n \delta)}.
\]
\end{itemize}
Then the stopping time $\tau_\ve = \inf \{ t \ge 0:\ g(X_t) \ge w(X_t) - \ve\}$ is $\ve$-optimal for any $\ve > 0$, the value function $w$ is continuous and satisfies, in a viscosity sense, the variational inequality
\[
\min \big( -\laa w + rw - f, w - g \big) = 0,
\]
where $\laa$ is the infinitesimal generator of $(X_t)$.

If, additionally, $R = \lim_{t \to \infty} \int_0^t r(X_s) ds < \infty$ $\prob^x$-a.s. (for example, $r \equiv 0$)  then the stopping time $\tau^* = \inf \{ t \ge 0:\ g(X_t) \ge w(X_t) \}$ is optimal for $w$.
\end{thm}
Examples of processes satisfying (i)-(iii) are given in Section \ref{sec:example}. They include a large class of (non-uniformly) exponentially or sub-exponentially ergodic diffusions (e.g., polynomially ergodic), including a prime example of an Ornstein-Uhlenbeck process. Assumption (i') is mostly applicable in the undiscounted case ($r \equiv 0$) to processes on compact state spaces (for example, for reflected diffusions). It then follows from the weak Feller property \ref{ass:weak_feller}.

Theorem \ref{thm:intro_undis} combines selected results from Sections \ref{sec:no_disc}-\ref{sec:VI}. For the undiscounted problem ($r \equiv 0$), conditions (i)-(ii) or (i') imply the main assumption \ref{ass:undis_large_dev} at the beginning of Subsection \ref{subsec:undis_general} through results in Subsections \ref{subsec:undis_potential} and \ref{subsec:large_dev}. Theorem \ref{thm:optimal_time_bounded_q} shows that $\tau^*$ is an optimal stopping time. Continuity of $w$ is proved in Subsection \ref{subsec:continuity}. The case with a general discount rate $r$ is studied in Section \ref{sec:general}. $\ve$-optimality of $\tau_\ve$ is demonstrated in Theorem \ref{thm:general_disc_optimality}. Theorem \ref{thm:gen_continuity} shows that $w$ is continuous. Optimality of $\tau^*$ follows from Theorem \ref{thm:general_optimality}. Variational characterisation of the value function is established in Section \ref{sec:VI}.

The structure of the paper is as follows. Section \ref{sec:no_disc} presents results for the undiscounted functional, $r \equiv 0$. These results are generalised to an arbitrary discount rate in Section \ref{sec:general}. Variational characterisation of the value function is presented in Section \ref{sec:VI}. In Section \ref{sec:dichotomy} we show that when $\mu(r) > 0$ and a certain  upper bound for the large deviations of the empirical process $r(X_t)$ holds, then the stopping problem \eqref{eqn:general} inherits many properties from the problem with a constant positive discount rate, in particular, the value function is finite for an arbitrary function $f$ (not necessarily satisfying $\mu(f) < 0$).

\section{Undiscounted functional}\label{sec:no_disc}
The study of the stopping problem \eqref{eqn:general} commences with a special case when the discount rate $r (x)$ is zero:
\begin{equation}\label{eqn:stopping_problem}
w(x) = \sup_{\tau} \limsup_{T \to \infty} \ee^x \Big\{ \int_0^{\tau \wedge T} f(X_s) ds + g(X_{\tau \wedge T}) \Big\}.
\end{equation}
This particular stopping problem will highlight main differences between a standard discounted functional and the functional with the discount rate that is not uniformly separated from zero. To appreciate the difficulty of the problem notice that by the law of large numbers for ergodic processess (see Lemma \ref{lem:LLN} below) we have
\[
\lim_{t \to \infty} \frac{1}{t} \int_0^t f^+(X_s) ds = \mu(f^+)
\qquad \text{and} \qquad
\lim_{t \to \infty} \frac{1}{t} \int_0^t f^-(X_s) ds = \mu(f^-)
\]
$\prob^x$-a.s. for all $x \in E$. When both the positive part $f^+$ of $f$ and the negative part $f^-$ of $f$ are $\mu$-a.s. non-zero then the integral $\int_0^\infty f(X_s) ds$ is undefined because it is a difference of two infinities. Ergodicity of the process $(X_t)$ introduces a delicate trade-off between these integrals ensuring that one infinity ``cancels'' the other.

\begin{lemma}\label{lem:LLN}
Under assumption \ref{ass:ergodic}, for any measurable function $h:E \to \er$ such that $\mu(|h|) < \infty$, we have for $\mu$ almost all $x \in E$
\[
\frac{1}{t} \int_0^t h(X_s) ds \to \mu(h) \qquad \text{$\prob^x$-a.s. as $t \to \infty$,}
\]
where $\mu(h)$ denotes $\int_E h(x) \mu(dx)$.
\end{lemma}
The proof of the above result follows from the law of large numbers for stationary processes \cite[Chapter X, Theorem 2.1]{Doob1953}; for details see \cite[Theorem 17.1.2]{Meyn2009}.

The solution of the stopping problem \eqref{eqn:stopping_problem} is determined by the value $\mu(f)$. The case $\mu(f) > 0$ is trivial: it is optimal \emph{not} to stop and the value function is infinite:
\begin{lemma}\label{lem:mu_f_positive}
Assume \ref{ass:ergodic}. If $\mu(f) > 0$ then $w(x) = \infty$ and it is optimal to take $\tau = \infty$,  $\prob^x$-a.s. for all $x \in E$.
\end{lemma}
\begin{proof}
The ergodicity of the process $(X_t)$ implies that $\lim_{t \to \infty} P_t f(x) = \mu(f) > 0$, so for $\tau = \infty$ we have
\[
\limsup_{T \to \infty} \ee^x \Big\{ \int_0^{T} f(X_s) ds + g(X_{T}) \Big\} \ge \limsup_{T \to \infty} \int_0^T P_s f(x) ds - \|g\|_\infty = \infty.
\]
\end{proof}
The case $\mu(f) < \infty$ is the main topic of this paper and results in a much more interesting behaviour. In particular, the value function is often finite. The case of $\mu(f) = 0$ is a grey area -- a transition point between interesting and trivial results. Our methods do not allow us to provide a thorough study of this case. However, a glimpse at possible results can be caught if one takes $f = 0$. Obviously, the value function is finite, but the existence of an optimal stopping time depends on the interplay between the process and the function $g$ as the following example shows.
\begin{example}
Let $(X_t)$ be an Ornstein-Uhlenbeck process
\[
dX_t = \theta (\mu - X_t) dt + \sigma dW_t.
\]
Take $f \equiv 0$ and $g(x) = \arctan(x)$. Obviously, $w(x) \le \max_x g(x) = \pi$. Take a stopping time $\tau_\ve = \inf \{ t \ge 0: X_t \ge K(\ve) \}$, where $K(\ve) = \tan (\pi - \ve)$. This stopping time is $\ve$-optimal because $g(X_{\tau_\ve}) \ge \pi - \ve \ge w(x) - \ve$. This also proves that $w(x) = \pi$. However, there is no optimal stopping time as the objective function $g(x)$ grows when $x$ increases.
\end{example}

\subsection{Optimal stopping times}\label{subsec:undis_general}
The developments of this subsection are pursued under the following additional assumption:
\begin{assumption}
\item[(B1)]\label{ass:undis_large_dev} For any $x \in E$, there is $d(x) < 0$ such that
\[
\gamma (x) = \sup_{\tau} \limsup_{T \to \infty} \ee^x \Big\{ \int_0^{\tau \wedge T} \big(f(X_s) - d(x)\big) ds \Big\} < \infty.
\]
\end{assumption}
Sufficient conditions will be provided in the following subsections.

Although the above assumption does not state explicitly that $\mu(f) <0$, this is a necessary condition. If $\mu(f) \ge 0$ then $\mu(f) - d(x) > 0$ and by ergodicity
\[
\lim_{t \to \infty} \ee^x \{ f(X_t) - d(x) \} = \mu(f) - d(x) > 0,
\]
which implies $\gamma(x) = \infty$.

\begin{remark}\label{rem:B1}
It is vital for applications to notice that if assumption \ref{ass:undis_large_dev} is satisfied for a function $f$, it is also satisfied for any function $f' \le f$. This trivial observation greatly extends the applicability of sufficient conditions for \ref{ass:undis_large_dev} discussed in Subsections \ref{subsec:undis_potential} and \ref{subsec:large_dev}.
\end{remark}

Assumption \ref{ass:undis_large_dev} allows us to restrict the set of stopping times in \eqref{eqn:stopping_problem} to those with finite expectation.
\begin{lemma}\label{lem:bound_for_stop}
For every integrable $\ve$-optimal stopping time $\sigma$ we have
\[
\ee^x \{ \sigma \} \le \frac{\gamma(x) + 2\|g\| + \ve}{-d(x)}.
\]
\end{lemma}
\begin{proof}
Since $\sigma$ is integrable, we have
\begin{multline*}
\ee^x \Big\{ \int_0^{\sigma} f(X_s)ds + g(X_{\tau})\Big\}
=
\ee^x \Big\{ \int_0^{\sigma} \big(f(X_s)-d(x)\big)ds + d(x) \sigma  + g(X_{\sigma})\Big\}\\
\le
\gamma(x) + \|g\| + d(x) \ee^x\{ \sigma\}.
\end{multline*}
On the other hand,
\[
\ee^x \Big\{ \int_0^{\sigma} f(X_s)ds + g(X_{\tau})\Big\} \ge w(x) - \ve \ge -\|g\| - \ve.
\]
Combining the above two inequalities,
\[
-\|g\| - \ve \le \gamma(x) + \|g\| + d(x) \ee^x\{ \sigma\},
\]
completes the proof.
\end{proof}

It follows from \eqref{eqn:stopping_problem} that the value function can be defined as a supremum over bounded (hence, integrable) stopping times. The above lemma proves that the optimisation can be restricted to those stopping times whose expectation is bounded by
\begin{equation}\label{eqn:Mx}
M(x) = \frac{\gamma(x) + 2\|g\| + 1}{-d(x)}.
\end{equation}
Hence, the value function $w(x)$ is bounded from above by $\|f\| M(x) + \|g\|$.

\begin{remark}\label{rem:integrable_tau}
By the dominated convergence theorem and quasi left-continuity of $(X_t)$ (see \cite[Theorem 3.13]{Dynkin1965}) whenever $\ee^x \{\tau\} < \infty$ we have
\[
\limsup_{T \to \infty} \ee^x \Big\{ \int_0^{\tau\wedge T} f(X_s) ds + g(X_{\tau \wedge T})  \Big\}
= \ee^x \Big\{ \int_0^\tau f(X_s) ds + g(X_\tau) \Big\}.
\]
\end{remark}

\begin{lemma}\label{lem:lsc_bounded_q}
The value function $w$ is lower semi-continuous.
\end{lemma}
\begin{proof}
Let
\begin{equation}\label{eqn:w_T}
w_T(x) = \sup_{\tau \le T} \ee^x \Big\{ \int_0^{\tau} f(X_s) ds + g(X_{\tau}) \Big\}.
\end{equation}
Due to the continuity and boundedness of functions $f$ and $g$ as well as the weak Feller property of the underlying process $(X_t)$, the value function $w_T$ is continuous, see \cite[Corollary 3.6]{Palczewski2008}.
Functions $w_T$ converge pointwise to $w$ and they form an increasing sequence. Hence, $w$ is lower semicontinuous.
\end{proof}

The construction of an optimal stopping time requires a few intermediate steps achieved in lemmas below.

\begin{lemma}\label{lem:undisc_sigma_m}
For every $x \in E$, there exists an non-decreasing sequence $\sigma_m$ of bounded $\frac1m$-optimal stopping times for $w(x)$.
\end{lemma}
\begin{proof}
The stopping problem $w(x)$ can be equivalently written as a supremum over bounded stopping times. Hence, $w_T(x)$ defined in \eqref{eqn:w_T} approximate $w(x)$ from below. By \cite[Corollary 3.6]{Palczewski2008}, these stopping problems admit optimal solutions of the form
\[
\tau_T = \inf \{ t \in [0, T]:\ g(X_t) \ge w_{T-t}(X_t) \}.
\]
As $w_T$ are non-decreasing in $T$, it is easy to see that $\tau_T$ are non-decreasing in $T$ as well. By setting $T(m) = \inf \{ T \ge 0:\ w(x) - w_T(x) \le \frac1m \}$, we can set $\sigma_m = \tau_{T(m)}$.
\end{proof}

In what follows we shall denote by $(\sigma_m)$ the sequence of stopping times from the above lemma.

\begin{lemma}\label{lem:super_mart}\label{lem:undisc_snell}
The process $Z_t := \int_0^t f(X_s) ds + w(X_t)$ is a right-continuous $\prob^x$-supermartingale for any $x \in E$. Moreover, for a bounded stopping time $\sigma$ and an arbitrary stopping time $\tau$
\begin{equation}\label{eqn:undisc_bellman_ineq}
\ee^{x} \Big \{ \int_0^\sigma f(X_s) ds + g(X_\sigma) \Big\}
\le \ee^{x} \Big\{ \int_0^{\tau \wedge \sigma} f(X_s) ds + \ind{\sigma < \tau} g(X_{\sigma}) + \ind{\sigma \ge \tau} w(X_{\tau}) \Big\}.
\end{equation}
\end{lemma}
\begin{proof}
Define $Z_t^T = \int_0^t f(X_s) ds + w_{T-t}(X_t)$, $t \in [0, T]$. This process is a Snell envelope of the process $t \mapsto \int_0^t f(X_s) ds + g(X_t)$, so, in particular, it is a right-continuous supermartingale \cite[Theorem 2.4]{Peskir2006}. These supermartingales are increasing in $T$ as so is $w_{T-t}$.  Notice that $Z_t = \sup_{T \ge t} Z_t^T$ $\prob^x$-a.s. Theorem T16 in \cite{Meyer1966} and the remark following it imply that $(Z_t)$ is a right-continuous supermartingale.

Since $\sigma$ is a bounded stopping time, the optional sampling theorem yields $\ee^x \{ Z_{\sigma} | \ef_{\tau \wedge \sigma} \} \le Z_{\tau \wedge \sigma}$. This reads
\[
\int_0^{\tau \wedge \sigma} f(X_s) ds + w(X_{\tau \wedge \sigma}) \ge \ee^{x} \Big \{ \int_0^\sigma f(X_s) ds + w(X_\sigma) \Big| \ef_{\tau \wedge \sigma} \Big\}.
\]
Hence,
\begin{align*}
\int_0^{\tau \wedge \sigma} f(X_s) ds + \ind{\sigma \ge \tau} w(X_{\tau \wedge \sigma})
&\ge
\ee^{x} \Big \{ \int_0^\sigma f(X_s) ds + \ind{\sigma \ge \tau} w(X_\sigma) \Big| \ef_{\tau \wedge \sigma} \Big\}\\
&\ge
\ee^{x} \Big \{ \int_0^\sigma f(X_s) ds + \ind{\sigma \ge \tau} g(X_\sigma) \Big| \ef_{\tau \wedge \sigma} \Big\}.
\end{align*}
Adding $\ind{\sigma < \tau} g(X_\sigma)$ to both sides completes the proof.
\end{proof}

Due to the lower semi-continuity of $w$ and the continuity of $g$, the random variable
\[
\tau_\ve = \inf \{t \ge 0:\ g(X_t) \ge w(X_t) - \ve \}
\]
is a stopping time and $g(X_{\tau_\ve})\ge w(X_{\tau_\ve})-\ve$ on $\{\tau_\ve < \infty\}$.

\begin{lemma}\label{lem:undisc_tech}
For any $x \in E$, we have $\displaystyle \prob^x \big\{ \sigma_m < \tau_\ve\big\} \le \frac{1}{m \ve}.$
\end{lemma}
\begin{proof}
Apply \eqref{eqn:undisc_bellman_ineq} to $\sigma_m$ and $\tau_\ve$
\begin{align*}
w(x) - \frac1m
&\le \ee^{x} \Big \{ \int_0^{\sigma_m} f(X_s) ds + g(X_\sigma) \Big\}\\
&\le \ee^{x} \Big\{ \int_0^{\sigma_m\wedge \tau_\ve}  f(X_s) ds + \ind{\sigma_m < \tau_\ve} g(X_{\sigma_m}) + \ind{\sigma_m \ge \tau_\ve} w(X_{\tau_\ve}) \Big\}\\
&\le  \ee^{x} \Big\{ \int_0^{\sigma_m\wedge \tau_\ve} f(X_s) ds + \ind{\sigma_m < \tau_\ve} \big( w(X_{\sigma_m}) - \ve \big) + \ind{\sigma_m \ge \tau_\ve} w(X_{\tau_\ve})\Big\}\\
&\le  \ee^{x} \Big\{ \int_0^{\sigma_m\wedge \tau_\ve} f(X_s) ds + w(X_{\sigma_m \wedge \tau_\ve}) \Big\} - \ve \prob^x \big\{\sigma_m < \tau_\ve\big\} \\
&\le w(x) - \ve \prob^x \big\{\sigma_m < \tau_\ve \big\},
\end{align*}
where the third inequality follows from the observation that $g(X_{\sigma_m}) < w(X_{\sigma_m}) - \ve$ on $\{ \sigma_m < \tau_\ve \}$ and the last inequality is the result of \eqref{eqn:undisc_bellman_ineq} applied for $\tau \equiv 0$ and $\sigma = \sigma_m\wedge \tau_\ve$.
\end{proof}

\begin{lemma}\label{lem:bound_expectation}
Let $(X_n)$ be a sequence of positive random variables with $\ee X_n \le 1$. If $Y$ is a positive random variable (with a possibly infinite expectation) and $\lim_{n \to \infty} \prob\{Y > X_n\} =0$ then $\ee Y \le 1$.
\end{lemma}
\begin{proof}
Assume first that $Y \le K$ a.s. for some constant $K > 0$. Then
\begin{align*}
\ee \{ Y\} &= \ee \{ \ind{Y \le X_n} Y \} + \ee \{ \ind{Y > X_n} Y \}\\
&\le \ee \{ \ind{Y \le X_n} X_n \} + \ee \{ \ind{Y > X_n} Y \}\\
&\le \ee \{ X_n \} + K \prob \{ Y > X_n\} \le 1 + K \prob \{ Y > X_n\}
\end{align*}
and the last term tends to $0$ as $n \to \infty$. Take now an arbitrary $Y$ satisfying assumptions of the lemma. For each $K > 0$, we have $\lim_{n \to \infty} \prob\{Y\wedge K > X_n\} = 0$. Previous considerations imply $\ee \{ Y \wedge K \} \le 1$. By monotone convergence, $\ee \{ Y \} = \lim_{K \to \infty} \ee \{ Y \wedge K \} \le 1$.
\end{proof}

\begin{thm}\label{thm:optimal_time_bounded_q}
Under assumptions \ref{ass:weak_feller}-\ref{ass:continuity} and \ref{ass:undis_large_dev} an optimal stopping time is given by
\[
\tau^* = \inf\{ t \ge 0:\ g(x) \ge w(x) \}.
\]
Moreover, $\ee^x \{ \tau^* \} \le M(x)$, where $M(x)$ is defined in \eqref{eqn:Mx}.
\end{thm}
\begin{proof}
By Lemma \ref{lem:bound_for_stop}, for every $m$ the stopping time $\sigma_m$ has a bounded expectation: $\ee^x \{ \sigma_m \} \le M(x)$, where $M(x)$ is defined in \eqref{eqn:Mx}. This bound does not depend on $m$. Combining this with conclusions of Lemma \ref{lem:undisc_tech} and Lemma \ref{lem:bound_expectation} (with $Y=\tau_\ve/M(x)$ and $X_n=\sigma_n/M(x)$) yields that $\ee^x \{ \tau_\ve \} \le M(x)$.

Apply \eqref{eqn:undisc_bellman_ineq} to $\sigma_m$ and $\tau_\ve$ and notice that $w(X_{\tau_\ve}) - \ve \le g(X_{\tau_\ve})$ due to the continuity of $g$, lower semicontinuity of $w$ and right-continuity of the process $(X_t)$:
\begin{align*}
w(x) - \frac1m
&\le \ee^{x} \Big \{ \int_0^{\sigma_m} f(X_s) ds + g(X_{\sigma_m}) \Big\}\notag\\
&\le \ee^{x} \Big\{ \int_0^{\sigma_m\wedge \tau_\ve} f(X_s) ds + \ind{\sigma_m < \tau_\ve} g(X_{\sigma_m}) + \ind{\sigma_m \ge \tau_\ve} w(X_{\tau_\ve}) \Big\}\\
&\le \ee^{x} \Big\{ \int_0^{\sigma_m\wedge \tau_\ve} f(X_s) ds + \ind{\sigma_m < \tau_\ve} g(X_{\sigma_m}) + \ind{\sigma_m \ge \tau_\ve} \big(g(X_{\tau_\ve}) + \ve \big)\Big\}.
\end{align*}
When $m\to \infty$, the dominated convergence theorem implies
\[
w(x) \le \ee^{x} \Big\{ \int_0^{\tau_\ve} f(X_s) ds + g(X_{\tau_\ve}) + \ve \Big\},
\]
so $\tau_\ve$ is $\ve$-optimal. Put $\tau_0 = \lim_{\ve \to 0} \tau_\ve$. Since this is an increasing sequence of random variables with the expectation bounded by $M(x)$ then $\tau_0$ is well defined and $\ee^x \{ \tau_0 \} \le M(x)$. For $\eta > 0$ and any $0 < \ve \le \eta$ we have $g(X_{\tau_\ve}) \ge w(X_{\tau_\ve}) - \eta$. By the quasi left-continuity of $(X_t)$ and lower semicontinuity of $w$, taking the limit $\ve \to 0$ and then $\eta \to 0$ yields
\[
g(X_{\tau_0}) \ge w(X_{\tau_0}).
\]
So $\tau_0 \ge \tau^*$. This means that $\tau^*$ is finite and $\ee^x \{ \tau^* \} \le M(x)$. Its optimality follows in the same way as $\ve$-optimality of $\tau_\ve$.
\end{proof}
The stopping time $\tau^*$ can be characterised as a limit of optimal stopping times for problems with a finite stopping horizon when the horizon increases to infinity.
\begin{lemma}
We have $\tau^* = \lim_{T \to \infty} \tau_T$, where $\tau_T$ is an optimal stopping time for the problem \eqref{eqn:w_T}.
\end{lemma}
\begin{proof}
Recall that $\tau_T = \inf \{ t \ge 0: g(X_t) \ge w_{T-t} (X_t) \}$. The sequence $(\tau_T)_{T \ge 0}$ of random variables is increasing  and each element has its expectation bounded by $M(x)$. Hence, $\tau_{\infty} = \lim_{T \to \infty} \tau_T$ exists.  Since $\tau_T \le \tau^*$, we conclude that $\tau_\infty \le \tau^*$. We also know that $g(X_{\tau_T}) = w_{T - \tau_T}(X_{\tau_T})$, so $\lim_{T \to \infty} g(X_{\tau_T}) = \lim_{T \to \infty} w_{T - \tau_T}(X_{\tau_T})$. By the quasi left-continuity of $(X_t)$ we have $\lim_{T \to \infty} g(X_{\tau_T}) = g(X_{\tau_\infty})$. For each $\omega \in \Omega$, any $T > \tau^*(\omega)$ and $0 < d < T - \tau_T(\omega)$ we have $w_{T-\tau_T}(X_{\tau_T})\ge w_d(X_{\tau_T})$ (dependence on $\omega$ is omitted here for clarity of notation).  Hence, for each positive $d$ we have $\liminf_{T\to \infty} w_{T-\tau_T}(X_{\tau_T})\ge  w_d(X_{\tau_\infty})$ $\prob^x$-a.s. and letting $d \to \infty$ we obtain  $\liminf_{T\to \infty} w_{T-\tau_T} (X_{\tau_T})\ge w(X_{\tau_\infty})$. Finally, $g(X_{\tau_\infty})\ge w(X_{\tau_\infty})$, $P^x$ - a.s., so $\tau_\infty \ge \tau^*$.
%
\end{proof}

Summarising, under the assumptions of this section we showed that the value function $w(x)$ is lower semi-continuous and the stopping time $\tau^*$, the first time the process enters the closed set $\{ x \in E:\ g(x) \ge w(x) \}$, has a finite expectation and is optimal.

\subsection{Zero-potential bounded from below}\label{subsec:undis_potential}
This and the following subsection provide sufficient conditions for \ref{ass:undis_large_dev}. Here, they are first expressed in terms of a zero-potential
\[
q(x) = \limsup_{T \to \infty} \ee^x \Big\{ \int_0^T \big(f(X_s) - \mu(f)\big) ds \Big\}.
\]
They are later derived from certain conditions related to the speed of convergence of transition probabilities of the process $(X_t)$ to the invariant measure $\mu$.

\begin{assumption}
\item[(C1)] $q(x)$ is continuous and bounded from below \label{ass:B1}
\item[(C2)] for any bounded stopping time $\sigma$
\[
q(x) = \ee^x \Big\{ \int_0^{\sigma} \big(f(X_s) - \mu(f) \big) ds + q(X_\sigma)\Big\}.
\] \label{ass:B2}
\item[(C3)] $\mu(f) < 0$ \label{ass:B3}
\end{assumption}
\begin{remark}
Assumption \ref{ass:B2} is equivalent to requesting that the process $Z_t = \int_0^t \big(f(X_s) - \mu(f) \big) ds + q(X_t)$ is a martingale.
\end{remark}

\begin{remark}
Under assumption \ref{ass:B2}, for any bounded stopping time $\tau$
\[
\ee^x \Big\{ \int_0^{\tau} f(X_s) ds + g(X_{\tau}) \Big\} = q(x) + \ee^x \{ \mu(f) \tau  + (g - q)(X_{\tau}) \}.
\]
The optimal stopping problem \eqref{eqn:stopping_problem} can be equivalently written as
\begin{equation}\label{eqn:w_hat}
\hat w(x) := w(x) - q(x) =  \sup_\tau \limsup_{T \to \infty} \ee^x \{  \mu(f) (\tau \wedge T) + (g - q)(X_{\tau \wedge T}) \}.
\end{equation}
This transforms the problem with a possibly unbounded integral term into a problem with an unbounded terminal reward and a strictly negative running cost.
\end{remark}
\begin{remark}
The limit $\limsup_{T \to \infty}$ can be omitted in the definition of $\hat w$ and $w$ and the set of stopping times can be restricted to those with a finite expectation. Indeed, if $\ee^x \{\tau \} = \infty$, then
\[
\limsup_{T \to \infty} \ee^x \{  \mu(f) (\tau \wedge T) + (g - q)(X_{\tau \wedge T}) \} \le \mu(f) \lim_{T \to \infty} \ee^x\{ \tau \wedge T \} + \Gamma = - \infty,
\]
where $\Gamma = \sup_{x} (g-q)(x) < \infty$ by assumption \ref{ass:B1}.
\end{remark}

The following lemma shows that the assumptions introduced at the beginning of this subsection are sufficient for \ref{ass:undis_large_dev}.
\begin{lemma}\label{lem:gamma_from_q}
Under \ref{ass:B1}-\ref{ass:B3}, assumption \ref{ass:undis_large_dev} is satisfied with any $d(x) \in (\mu(f), 0)$. Moreover, $\gamma(x) \le q(x) - A$, where $A = \inf_y q(y)$.
\end{lemma}
\begin{proof}
Take any stopping time $\tau$ and $T > 0$. Then
\begin{align*}
&\ee^x \Big\{ \int_0^{\tau \wedge T} \big(f(X_s) - d(x) \big) ds \Big\}\\
&=
\ee^x \Big\{ \int_0^{\tau \wedge T} \big(f(X_s) - \mu(f) \big) ds + \big(\mu(f) - d(x) \big) (\tau \wedge T)\Big\}\\
&\le
\ee^x \Big\{ \int_0^{\tau \wedge T} \big(f(X_s) - \mu(f) \big) ds \Big\}\\
&=
q(x) - \ee^x\{ q(X_{\tau \wedge T}) \} \le q(x) - A,
\end{align*}
where $A$ is a lower bound for $q$ and in the first inequality we used $\mu(f) - d(x) < 0$. From the arbitrariness of $\tau$ and $T$ we conclude that $\gamma(x) \le q(x) - A < \infty$.
\end{proof}

The existence and continuity of the zero-potential is ensured by \ref{ass:speed} below, while the integrability required in \ref{ass:B2} follows from \ref{ass:integrability}.
\begin{assumption}
\item[(D1)] \label{ass:speed} there is $K: E \to (0, \infty)$ bounded on compacts and $h:[0, \infty) \to \er_+$ such that
\[
\| P_t(x, \cdot) - \mu(\cdot) \|_{TV} < K(x) h(t),
\]
and $\int_0^\infty h(t) dt < \infty$,
\item[(D2)] \label{ass:integrability} (integrability) $\ee^x \{ K(X_T) \} < \infty$ for any $x \in E$ and $T \ge 0$.
\end{assumption}
\begin{remark} Ergodicity assumption \ref{ass:speed} is commonly studied in the theory of stochastic processes (see, e.g., references in Section \ref{sec:example}):
\begin{enumerate}
\item (geometric ergodicity) there are functions $K, \lambda: E \to (0, \infty)$ bounded on compacts such that
\[
\| P_t(x, \cdot) - \mu(\cdot) \|_{TV} < K(x) e^{-\lambda(x) t},
\]
\item (polynomial ergodicity) there is a function $K: E \to (0, \infty)$ bounded on compacts and a constant $p > 1$ such that
\[
\| P_t(x, \cdot) - \mu(\cdot) \|_{TV} < \frac{K(x)}{(1 + t)^p}.
\]
\end{enumerate}
\end{remark}

\begin{lemma}\label{lem:integrability}
Assume \ref{ass:speed} and \ref{ass:integrability}. Then the zero-potential $q$ is finite-valued and continuous, and assumption \ref{ass:B2} is satisfied.
\end{lemma}
\begin{proof}
Define
\[
\bar q(x) = \Big\{ \int_0^\infty \big|P_s f(x) - \mu(f)\big| ds \Big\}.
\]
This function is finite-valued by assumption \ref{ass:speed}. This means that $q(x)$ is finite valued and continuous by the weak Feller property (and dominated convergence theorem). By \ref{ass:integrability}, for any $T > 0$ and $x \in E$, the random variable $\bar q(X_T)$ is $\prob^x$-integrable.

Denote $\hat f(x) = f(x) - \mu(f)$. For $\alpha > 0$ and $T > 0$, due to the boundedness of $\hat f$ and Markov property of $(X_t)$ we have
\begin{equation}\label{eqn:discounted_eqn}
\ee^x \Big\{ \int_0^T e^{-\alpha s} \hat f(X_s) ds \Big\} = q_\alpha(x) - e^{-\alpha T} \ee^x \{ q_\alpha(X_T) \},
\end{equation}
where
\[
q_\alpha(x) = \ee^x \Big\{ \int_0^\infty e^{-\alpha s} \hat f(X_s) ds \Big\}.
\]
Clearly, $|q_\alpha(x)| \le \bar q(x)$, so $|q_\alpha(X_T)|$ is dominated by $\bar q(X_T)$. The integrability of the latter and the dominated convergence theorem yield that when $\alpha \to 0$ the equation \eqref{eqn:discounted_eqn} converges to
\[
\ee^x \Big\{ \int_0^T \hat f(X_s) ds \Big\} = q(x) - \ee^x \{ q(X_T) \}.
\]

Take now a stopping time $\sigma$ bounded by $T$. Then
\begin{align*}
|q_\alpha(X_\sigma)|
&= \Big| \ee^x \Big\{ \int_\sigma^\infty e^{-\alpha (s-\sigma)} \hat f(X_s) ds \Big| \ef_\sigma \Big\} \Big|\\
&\le \Big| \ee^x \Big\{ \int_\sigma^T e^{-\alpha (s-\sigma)} \hat f(X_s) ds \Big| \ef_\sigma \Big\} \Big|
+ \Big| e^{-\alpha(T-\sigma)} \ee^x \Big\{ \int_T^\infty e^{-\alpha (s-T)} \hat f(X_s) ds \Big| \ef_\sigma \Big\} \Big|\\
&\le T \|\hat f\| + \Big| \ee^x \Big\{ e^{-\alpha(T-\sigma)} \ee^x \big\{ \int_T^\infty e^{-\alpha (s-T)} \hat f(X_s) ds \big| \ef_T \big\} \Big| \ef_\sigma \Big\} \Big|\\
&= T \|\hat f\| + e^{-\alpha(T-\sigma)} \Big| \ee^x \big\{  q_\alpha(X_T) \big| \ef_\sigma \big\} \Big|
\le T \|\hat f\| + \Big| \ee^x \big\{  \bar q(X_T) \big| \ef_\sigma \big\} \Big|,
\end{align*}
which gives a bound for $|q_\alpha (X_\sigma)|$ independent of $\alpha$. Strong Markov property gives
\begin{equation}\label{eqn:q_alpha}
\ee^x \Big\{ \int_0^\sigma e^{-\alpha s} \hat f(X_s) ds \Big\} = q_\alpha(x) -  \ee^x \{ e^{-\alpha \sigma} q_\alpha(X_\sigma) \}.
\end{equation}
Using the bound obtained above, we can take the limit $\alpha \to 0$ on both sides and conclude the proof.
\end{proof}

To ensure that $q$ is bounded from below, we have to assume something about the function $f$. Knowing that under ergodicity, the process spends most of the time in a ball with an appropriately big radius, we want to allow for flexibility of the function $f$ there. However, to give us control of the integral the effect of $f$ should be under control outside of the ball. The following lemma shows that it is sufficient to assume that $f > \mu(f)$ outside of a compact set.

\begin{lemma}\label{lem:f_C0}
Assume \ref{ass:speed}-\ref{ass:integrability}. If $\mu(f) < 0$ and the set $L = \{ x \in E:\  f(x) \le \mu(f) \}$ is compact, then $q$ is bounded from below.
\end{lemma}
\begin{proof}
Let $B$ be a closed ball such that $\mu(B) > 0$. It is easy to see (for example, by the law of large numbers for ergodic processes) that $\tau_B < \infty$ $\prob^x$-a.s. for any $x \in E$, where $\tau_B = \inf\{t \ge 0:\ X_t \in B \}$. Let $\tl B$ be the smallest closed ball containing $B$ and $L$. Then $\tau_{\tl B} < \infty$ $\prob^x$-a.s. for $x \in E$. We have $f(x) - \mu(f) > 0$ for $x \in \tl B^c$. For any $T>0$, equation \eqref{eqn:q_alpha} reads
\begin{align*}
q_\alpha(x)
&= \ee^x \Big\{ \int_0^{\tau_{\tl B} \wedge T} e^{-\alpha s} \big( f(X_s) - \mu(f)\big) ds \Big\} + \ee^x \{ e^{-\alpha (\tau_{\tl B} \wedge T)} q_\alpha(X_{\tau_{\tl B} \wedge T}) \}\\
&\ge \ee^x \{ e^{-\alpha (\tau_{\tl B} \wedge T)}q_\alpha(X_{\tau_{\tl B} \wedge T}) \}.
\end{align*}
By the dominated convergence theorem and finiteness of $\tau_{\tl B}$ we obtain
\begin{equation}\label{eqn:q_A_ineq}
q_\alpha(x) \ge \ee^x \{ e^{-\alpha \tau_{\tl B}} q_\alpha(X_{\tau_{\tl B}}) \}.
\end{equation}
Considerations in the proof of Lemma \ref{lem:integrability} imply that $q_\alpha(x)$ converges to $q(x)$ when $\alpha \downarrow 0$, but they do not apply to the right-hand side of the above inequality since $\tau_{\tl B}$ does not have to be bounded. However, $X_{\tau_{\tl B}} \in \tl B$ by the right-continuity of $(X_t)$, so it is sufficient to prove uniform convergence of $q_\alpha$ to $q$ on compact sets. We have the following estimate:
\begin{align*}
|q(x) - q_\alpha(x)|
&\le \int_0^\infty \big(1 - e^{-\alpha s}) | \ee^x \{ f(X_s) - \mu(f) \} ds\\
&\le \int_0^\infty \big(1 - e^{-\alpha s}) \|f\|  \| P_t(x, \cdot) - \mu(\cdot) \|_{TV}ds\\
&\le \|f\| K(x) \int_0^\infty \big(1 - e^{-\alpha s})  h(s) ds.
\end{align*}
By the dominated convergence theorem, this proves uniform convergence of $q_\alpha$ to $q$ as $\alpha \downarrow 0$ on compact sets since $K$ is bounded on compacts. Again, by the dominated convergence theorem, \eqref{eqn:q_A_ineq} yields
\[
q(x) \ge \ee^x \{ q(X_{\tau_{\tl B}}) \}.
\]
The function $q$ is continuous hence bounded on $\tl B$, so $q(x) \ge \inf_{y \in \tl B} q(y) > -\infty$.
\end{proof}
\begin{remark}
It is easy to see that the set $L$ from the above lemma is compact when $f \in \mathcal{C}_0$. This compactness property is also true for any continuous function $f$ with $\mu(f) < 0$ and bounded from below by a $\mathcal{C}_0$ function $\phi$. Indeed, $L \subset \{ x \in E:\  \phi(x) \le \mu(f) \} = L_\phi$ and the set $L_\phi$ is compact since $\mu(f) < 0$. Examples include functions $f$ that are positive outside of a sufficiently large ball but with negative values in a set of a large $\mu$-measure.
\end{remark}

It is not clear yet that $q$ can be unbounded from above and bounded from below. The lemma below shows that this is natural for a large class of processes.

\begin{lemma}
Assume \ref{ass:speed}-\ref{ass:integrability}. If $\mu(f) < 0$ and there is $\delta > 0$ such that the set $L_\delta = \{ x \in E:\  f(x) \le \mu(f) + \delta\}$ is compact, then $\lim_{\|x\| \to \infty} q(x) = \infty$ if and only if $\lim_{\|x\| \to \infty} \ee^x \{ \tau_{B_n} \} = \infty$ for all $n$, where $B_n = \{ x \in E:\ \|x\| \le n \}.$
\end{lemma}
\begin{proof}
Assume $\lim_{\|x\| \to \infty} \ee^x \{ \tau_{B_n} \} = \infty$ for all $n$. Let $N$ be such that $L_\delta \subset B_N$. As in the proof of Lemma \ref{lem:f_C0}, for any $\alpha > 0$
\begin{align*}
q_\alpha(x)
&= \ee^x \Big\{ \int_0^{\tau_{B_N}} e^{-\alpha s} \big( f(X_s) - \mu(f)\big) ds \Big\} + \ee^x \{ e^{-\alpha \tau_{B_N}} q_\alpha(X_{\tau_{B_N}}) \}\\
&> \delta \ee^x \{ \int_0^{\tau_{B_N}} e^{-\alpha s} ds\} + \ee^x \{ e^{-\alpha \tau_{B_N}} q_\alpha(X_{\tau_{B_N}}) \}.
\end{align*}
By the monotone convergence theorem $\lim_{\alpha \downarrow 0} \ee^x \{ \int_0^{\tau_{B_N}} e^{-\alpha s} ds\} = \ee^x \{ \tau_{B_N} \}$. In the proof of Lemma \ref{lem:f_C0} we have shown that $q_\alpha$ converges to $q$ uniformly on compact sets. Hence,
\[
q (x) \ge \delta \ee^x \{ \tau_{B_N} \} + \inf_{y \in B_N} q(y)
\]
and the proof is complete.

Assume now that $\lim_{\|x\| \to \infty} q(x) = \infty$ and take any $N > 0$. Then for $x \notin B_N$
\[
q_\alpha(x)
\le 2\|f\| \ee^x \bigg\{ \int_0^{\tau_{B_N}} e^{-\alpha s} ds\bigg\} + \ee^x \Big\{ e^{-\alpha \tau_{B_N}} q_\alpha(X_{\tau_{B_N}}) \Big\},
\]
which, when $\alpha \downarrow 0$, yields
\[
q (x) \le 2\|f\| \ee^x \{ \tau_{B_N} \} + \sup_{y \in B_N} q(y).
\]
Hence, $\lim_{\|x\| \to \infty} \ee^x \{ \tau_{B_N} \} = \infty$.
\end{proof}
Notice that the compactness of $L_\delta$ is not required in the proof of the left implication.

The above lemma suggests another criterion for the boundedness of $q$ from below. The proof is omitted.
\begin{lemma}
If $q$ is continuous and $q(x) \to \infty$ when $\|x\| \to \infty$, then $q$ is bounded from below.
\end{lemma}

\subsection{Large deviations type condition}\label{subsec:large_dev}

The second criterion for \ref{ass:undis_large_dev} is provided by an upper bound for large deviations of the empirical process:
\begin{assumption}
\item[(L)] For any $\delta, \ve>0$ and a compact set $K \subset E$ there is $N>0$ and $p>0$ such that for $n\ge N$ and $x \in K$ we have
\[
\prob^x\Big\{\Big|\frac1{n \delta} \int_0^{n \delta} f(X_s)ds - \mu(f)\Big|> \ve \Big\} \le e^{-p (n \delta)}.
\]\label{ass:L}
\end{assumption}

There are two important classes of Markov processes that satisfy assumption \ref{ass:L}:
\begin{itemize}
\item weakly Feller processes over a compact state space with a unique invariant measure (see \cite[Theorem 3]{DPDS} and also \cite{DV1, DV2}),
\item weakly Feller processes on a locally compact separable matric space with a unique invariant measure for which there exists a Lyapunov function, i.e., a function $u\ge 1$ such that for each positive $m$ the set $\{x \in E:\ \frac{u(x)}{P_\delta u(x)} \le m\}$ is compact (for details see \cite[Theorem 3]{DPDS} and Lemmas 4.1, 4.3 and Theorem 4.4 of \cite{DV2} and also \cite{DV1} and \cite{Liptser1997}).
\end{itemize}
Note that the first of the above condition means that if the state space $E$ is compact then our standing assumptions \ref{ass:weak_feller} and \ref{ass:ergodic} are sufficient for \ref{ass:L} to hold.

\begin{lemma}\label{lem:supremum_large_dev}
Assume \ref{ass:L}. For any $\ve > 0$ and a compact set $K \subset E$ there is $S>0$, $C>0$ and $\rho >0$ such that for $t\ge S$ and $x \in K$ we have
\[
\prob^x\Big\{ \sup_{s\geq t} \Big|\frac1s \int_0^s f(X_u)du-\mu(f)\Big| >\ve \Big\}\le Ce^{-\rho t}.
\]
\end{lemma}
\begin{proof}
Fix a compact set $K \subset E$ and put $\delta = \ve / (2 \|f - \mu(f)\|)$. Take $N$ and $p$ from (assumption \ref{ass:L}) such that for $n \ge N$ and $x \in K$ we have
\[
\prob^x\Big\{\Big|\frac1{n \delta} \int_0^{n \delta} f(X_s)ds - \mu(f)\Big|> \ve/2 \Big\} \le e^{-p (n \delta)}.
\]
The boundedness of $f$ and the choice of $\delta$ imply that for $n \ge N$
\[
\Big\{ \sup_{s\ge {n \delta}} \Big|\frac1s \int_0^s f(X_u)du-\mu(f)\Big| >\ve \Big\} \subset
\Big\{ \sup_{i \in \mathbb{N}} \Big|\frac1{(n + i) \delta} \int_0^{(n + i) \delta} f(X_u)du-\mu(f)\Big| >\ve/2 \Big\},
\]
where $\mathbb{N}$ denotes the set of non-negative integers. The probability of the event on the right-hand side is bounded from above by
\[
\sum_{i=0}^\infty \prob^x \Big\{\Big|\frac1{(n + i) \delta} \int_0^{(n + i) \delta} f(X_s)ds - \mu(f)\Big|>\ve / 2 \Big\}
\le
\sum_{i=0}^\infty e^{-p ((n + i) \delta)} = e^{-p (n \delta)} \frac{1}{1 - e^{-p\delta}}.
\]
This proves the statement of the lemma for $t = n \delta$, $n \ge N$, and $\rho = p$. Extending it to any $t \ge N\delta=: S$ is straightforward due to the boundedness of $f$.
\end{proof}
Notice that $\rho$ in the above lemma corresponds to $p$ in assumption \ref{ass:L} for $\ve/2$.

\begin{lemma}\label{lem:ldev_bound}
Assumption (L) and $\mu(f) < 0$ imply \ref{ass:undis_large_dev}.
\end{lemma}
\begin{proof}
Choose $0 < \ve < -\mu(f)$ and fix $x \in E$. By Lemma \ref{lem:supremum_large_dev} there is $S > 0$ such that for $t \ge S$ we have $\prob^x \{ A_t \} \le Ce^{-\rho t}$, where
\[
A_t = \Big\{ \sup_{s \ge t} \Big|\frac1s \int_0^s f(X_u)du-\mu(f)\Big| >\ve \Big\}.
\]
Then for any integrable stopping time $\tau$ we have
\begin{align*}
\ee^x\left\{\int_0^\tau (f(x_s)-\mu(f)) ds\right\}
&= \sum_{i=0}^\infty\ee^x \left\{1_{i\leq \tau <i+1} \int_0^\tau (f(x_s)-\mu(f))ds\right\}\\
&\le
\ve \ee^x\tau +2(\lfloor S \rfloor + 1)\|f\| +\sum_{i=\lfloor S \rfloor + 1}^\infty 2\|f\| (i+1) \prob^x\{A_i\}\\
&\le
\ve \ee^x\tau +2(\lfloor S \rfloor + 1)\|f\|+ \sum_{i=\lfloor S \rfloor + 1}^\infty 2\|f\| (i+1) Ce^{-i\rho}\\
&=\ve \ee^x \{\tau \} + \bar{C},
\end{align*}
This is equivalent to
\[
\ee^x\left\{\int_0^\tau \big(f(x_s)-\mu(f)-\ve\big) ds\right\} \le \bar C.
\]
Hence, \ref{ass:undis_large_dev} holds with $d(x) = \mu(f) + \ve$ which, by the choice of $\ve$, is strictly negative.
\end{proof}

\begin{remark}\label{rem:bound_Mx_largedev}
The above lemma implies that on any compact set one can keep $d(x)$ in assumption \ref{ass:undis_large_dev} constant and the resulting function $\gamma(x)$ is also bounded on this compact set. Hence, $M(x)$, defined in \eqref{eqn:Mx}, is bounded on compact sets.
\end{remark}

The estimate in assumption \ref{ass:L} is uniform over compact subsets of $E$. This is an excessive requirement for deducting \ref{ass:undis_large_dev}, but it will be needed for proving the continuity of the value function in the following section. Moreover, to the best of our knowledge most conditions for a weakly Feller process to satisfy \ref{ass:L} pointwise automatically imply that the bound is uniform on compact sets.

\subsection{Continuity of the value function}\label{subsec:continuity}
So far it has only been proved that the value function is lower semi-continuous, which was sufficient for contructing an optimal stopping time. However, it is often important to know if the value function is continuous. We will require it in Section \ref{sec:VI} to characterise the value function as a viscosity solution to an appropriate variational inequality.

\subsubsection{Assumption \ref{ass:L}}
Assumption \ref{ass:L} ensures the continuity of $w$ without any additional conditions. Denote by $\laa(x)$ the set of stopping times with the expectation bounded by $M(x)$, defined in \eqref{eqn:Mx}. Then
\begin{equation}\label{eqn:diff_w_wT}
\begin{aligned}
0 \le w(x) - w_T(x)
&\le
\sup_{\tau \in \laa(x)} \ee^x \Big\{ \int_{T}^{\tau \vee T} f(X_s) ds + \ind{\tau > T} \big( g(X_\tau) - g(X_T) \big) \Big\}\\
&\le
\sup_{\tau \in \laa(x)} \ee^x \Big\{ \int_T^{\tau \vee T} f(X_s) ds \Big\} + 2 \|g\| \frac{M(x)}{T},
\end{aligned}
\end{equation}
where $w_T$ is a value function for the stopping problem with horizon $T$.  Fix a compact set $L \subset E$. The second term converges to $0$ uniformly on $L$ since $M(x)$ is bounded on $L$ (see Remark \ref{rem:bound_Mx_largedev}). Take any $\ve > 0$. By Lemma \ref{lem:supremum_large_dev} there is $S > 0$ such that for $t \ge S$ we have $\prob^x \{ A_t \} \le Ce^{-\rho t}$, where
\[
A_t = \Big\{ \sup_{s \ge t} \Big|\frac1s \int_0^s f(X_u)du-\mu(f)\Big| >\ve \Big\}.
\]
Since $\mu(f) < 0$, for any stopping time $\tau \in \laa(x)$ we have
\begin{align*}
\ee^x \Big\{ \int_T^{\tau \vee T} f(X_s) ds \Big\}
&=
\ee^x\left\{\ind{\tau \ge T} \int_T^\tau (f(x_s)-\mu(f)) ds + \ind{\tau \ge T} \mu(f) (\tau-T) \right\}\\
&\le
\ee^x\left\{\ind{\tau \ge T} \int_T^\tau (f(x_s)-\mu(f)) ds\right\}.
\end{align*}
Then, for $T \ge S$:
\begin{align*}
&\ee^x\left\{\ind{\tau \ge T} \int_T^\tau (f(x_s)-\mu(f)) ds\right\}
\le
\sum_{i=0}^\infty \ee^x \Big\{1_{T+i\le \tau < T+i+1} \int_T^\tau (f(x_s)-\mu(f))ds\Big\}\\
&\le
\sum_{i=0}^\infty \ee^x \big\{1_{T+i\le \tau < T+i+1} \ind{A_{T+i}} (i+1) 2 \|f\| \big\}
+ \sum_{i=0}^\infty \ee^x \Big\{1_{T+i\le \tau < T+i+1} \ind{A^c_{T+i}} \int_T^\tau (f(x_s)-\mu(f))ds\Big\}\\
&\le
2 \|f\| C \sum_{i=0}^\infty (i+1) e^{-\rho (T+i)} +
\sum_{i=0}^\infty \ee^x \Big\{1_{T+i\le \tau < T+i+1} \ind{A^c_{T+i}} \Big[ \ve \tau + \Big| \int_0^T (f(x_s)-\mu(f))ds \Big| \Big] \Big\}.
\end{align*}
The first term is bounded by $2 \|f\| C e^{-\rho T} (1 - e^{-\rho})^{-2}$. For the second term notice that $A^c_{T} \subset A^c_{T + i}$. Hence,
it is bounded from above by
\[
\sum_{i=0}^\infty \ee^x \big\{1_{T+i\le \tau < T+i+1} \ind{A^c_{T+i}} [ \ve \tau + \ve T] \big\} \le 2\ve \ee^x \{\tau\} \le 2 \ve M(x).
\]
In total, we have demonstrated that
\[
0 \le w(x) - w_T(x) \le 2 \|f\| C e^{-\rho T} \frac{1}{(1 - e^{-\rho})^2} + 2 \ve M(x) + 2 \|g\| \frac{M(x)}{T}
\]
for sufficiently large $T$. This proves that when $T \to \infty$ the difference $w - w_T$ converges, uniformly on $L$, to a positive function that is bounded by $2 \ve M(x)$. Due to the arbitrariness of $\ve >0$, this proves that $w_T$ converges to $w$ uniformly on $L$. Recalling that the value functions $w_T$ are continuous yields the continuity of $w$.

\subsubsection{Negative $f$}
Assume that $f \le 0$. Recall that
\[
w(x) = \sup_{\tau \in \laa(x)} \ee^x \Big\{ \int_0^\tau f(X_s) ds + g(X_\tau) \Big\},
\]
where $\laa(x)$ is a set of stopping times $\tau$ such that $\ee^x \{ \tau \} \le M(x)$. Let $w_T(x)$ be the value function of the stopping problem with stopping times bounded by $T$. Then
\begin{align*}
0 \le w(x) - w_T(x)
&\le
\sup_{\tau \in \laa(x)}\ee^x \Big\{ \int_T^{T \vee \tau} f(X_s) ds + g(X_\tau) - g(X_{\tau \wedge T})\Big\}\\
&\le \sup_{\tau \in \laa(x)}\ee^x \{ \ind{\tau > T} 2 \|g\| \}\\
&\le 2 \|g\| \frac{M(x)}{T},
\end{align*}
where the penultimate inequality is due to $f \le 0$ and the last inequality follows from Tchebyshev's inequality. The function $M(x)$ is bounded on compacts when (a) assumption \ref{ass:L} holds (see Remark \ref{rem:bound_Mx_largedev}), or when (b) assumptions \ref{ass:B1}, \ref{ass:B2} are satisfied: by Lemma \ref{lem:gamma_from_q} $\gamma(x) \le q(x) - A$, where $A$ is the lower bound for $q$, and $q$ is continuous. This means that functions $w_T$ converge to $w$ uniformly on compact sets. As $w_T$ are continuous, so is $w$.

\subsubsection{Strong Feller property}
A Markov process $(X_t)$ satisfies the \emph{strong Feller property} if its semigroup $P_t$, $t > 0$, maps measurable bounded functions into continuous functions. The family of strong Feller processes includes uniformly non-degenerate diffusions, and jump-diffusions with uniformly non-degenerate diffusion term as well as solutions to stochastic differential equations driven by Levy processes with a non-zero Brownian part or with an appropriately high intensity of small jumps.

Under assumptions \ref{ass:B1}-\ref{ass:B2}, the optimal stopping problem \eqref{eqn:stopping_problem} is equivalent to the stopping problem \eqref{eqn:w_hat} which we recall for convenience:
\[
\hat w(x) := w(x) - q(x) =  \sup_\tau \limsup_{T \to \infty} \ee^x \{  \mu(f) (\tau \wedge T) + (g - q)(X_{\tau \wedge T}) \}.
\]
The function $\hat w$ is bounded from above because $\mu(f) <0$ and $q$ is bounded from below. Let $\delta > 0$. By the strong Feller property the mapping $x \mapsto \ee^x \{ \hat w(X_\delta) \vee (-M) \}$ is continuous for every $M>0$. By monotone convergence theorem, $\lim_{M \to \infty} \ee^x \{ \hat w(X_\delta) \vee (-M) \} = \ee^x \{ \hat w(X_\delta) \}$, so the function $P_\delta \hat w(x)$ is upper semicontinuous. We shall show now that $P_\delta \hat w$ converges to $\hat w$ uniformly on compact sets as $\delta \to 0$, which implies that $\hat w$ and $w = \hat w + q$ are upper semicontinuous. Combined with the lower semicontinuity of $w$ and the continuity of $q$ this yields that $w$ is continuous.

For the proof that $P_\delta \hat w$ converges to $\hat w$ uniformly on compact sets we need the following properties of weakly Feller processes.
\begin{lemma}\label{lem:compactness_of_Feller}(\cite[Proposition 2.1]{Palczewski2008})
For any compact set $K \subset E$, $T > 0$ and $\ve > 0$ there is $N$ such that for all $n \ge N$
\[
\sup_{x \in K} \prob^x \big\{ \forall t \in [0, T]\ X_t \in B(0,n) \big\} \ge 1- \ve,
\]
where $B(y,n)$ is a closed ball of radius $n$ centered at $y$.
\end{lemma}

\begin{lemma}\label{lem:continuity_of_Feller} (\cite[Theorem 3.7]{Dynkin1965})
For any compact set $K \subset E$ and any $\ve, \eta > 0$ there is $h_0 > 0$ such that
$$
\sup_{0 \le h \le h_0} \ \sup_{x \in K}\, \prob^x \{ X_h \notin B(x, \eta)\} < \ve.
$$
\end{lemma}

Notice that $\hat w(x) - P_\delta \hat w(x) = w(x) - P_\delta w(x) - q(x) + P_\delta q(x)$. Assumption \ref{ass:B2} gives
\[
P_\delta q(x) - q(x) = \ee^x \{ \int_0^\delta \big( f(X_s) - \mu(f) \big) ds \}.
\]
Hence, $|P_\delta q(x) - q(x)| \le 2\|f\|\delta$. By Lemma \ref{lem:super_mart} we have $\ee^x \{ w(X_\delta) + \int_0^\delta f(X_s) ds \} \le w(x)$ which implies
\[
P_\delta w(x) - w(x) \le \delta \|f\|.
\]
On the other hand, Lemma \ref{lem:undisc_snell} applied for $\sigma = \tau^*\wedge T$ (with $T \ge \delta$) and $\tau = \delta$ yields
\begin{align*}
w(x)
& = \limsup_{T \to \infty} \ee^x \Big\{ \int_0^{\tau^* \wedge T} f(X_s) ds + g(X_{\tau^*\wedge T}) \Big\}\\
&\le \ee^x \Big\{ \int_0^{\tau^* \wedge \delta} f(X_s) ds + \ind{\tau^* < \delta} g(X_{\tau^*}) + \ind{\tau^* \ge \delta} w(X_\delta) \Big\}.
\end{align*}
Hence,
\begin{align*}
w(x) - P_\delta w(x)
&\le
\delta \|f\| + \ee^x \{ \ind{\tau^* < \delta} \big[ g(X_{\tau^*}) - w(X_\delta) \big] \}\\
&\le
\delta \|f\| + \ee^x \{ \ind{\tau^* < \delta} \big[ g(X_{\tau^*}) - g(X_\delta) \big] \}\\
&\le
\delta \|f\| + \sup_{\tau \le \delta} \ee^x \{ g(X_{\tau}) - g(X_\delta) \}.
\end{align*}
Fix a compact set $K \subset E$. By Lemma \ref{lem:compactness_of_Feller}, for any $\ve > 0$ there is a compact set $L$ such that
\[
\sup_{x \in K} \prob^x \{ X_t \in L \ \forall\ t \in [0,1] \} \ge 1 - \ve.
\]
By Lemma \ref{lem:continuity_of_Feller}, for any $\eta > 0$ there is $h_0>0$ such that for all $\delta < h_0$ we have
\[
\sup_{0 \le h \le h_0} \ \sup_{x \in L}\, \prob^x \{ X_h \notin B(x, \eta)\} < \ve.
\]
Take $\delta < h_0 \wedge 1$. Then
\begin{align*}
&\sup_{\tau \le \delta} \ee^x \{ g(X_{\tau}) - g(X_\delta) \}\\
&\le \sup_{\tau \le \delta} \ee^x \big\{ \ind{X_\tau \in L} \ee^{X_\tau} \{ g(X_0) - g(X_{\delta-\tau}) \} + \ind{X_\tau \notin L} \big(g(X_{\tau}) - g(X_\delta)\big)\big\}\\
&\le \big(\omega(\eta) + \ve 2 \|g\| \big) + \ve 2 \|g\|,
\end{align*}
where $\omega_L(\cdot)$ is the modulus of continuity of $g$ restricted to the $1$-neighbourhood of $L$, i.e., the set $\{x \in E: \|x - y\|\le 1 \text{ for some $y \in L$}\}$. Summarising, for any $\eta \in (0,1)$ and a sufficiently small $\delta > 0$
\[
-\delta \|f\| \le w(x) - P_\delta w(x) \le \delta \|f\| + \omega_L(\eta) + 4 \ve \|g\|.
\]
Hence, $P_\delta(x)$ converges to $w(x)$ uniformly on $K$. This completes the proof of uniform on compact sets convergence of $P_\delta \hat w$ to $\hat w$ when $\delta \to 0$.

\subsubsection{Uniform integrability of $K(X_T)$}
Assume \ref{ass:speed} and
\begin{assumption}
\item[(D2')]\label{ass:uniform_int} for every compact set $L \subset E$ there is $\alpha > 0$ such that
\[
 \sup_{x \in L,\ T \ge 0} \ee^x \{ K(X_T)^{1+\alpha} \} < \infty,
\]
where $K$ is defined in assumption \ref{ass:speed}.
\end{assumption}
This, in particular, implies \ref{ass:integrability}. In the proof of Lemma \ref{lem:integrability} we have established that $|q(x)| \le \bar q(x) \le C K(x)$, where $C = 2 \|f\| \int_0^\infty h(t) dt$. Hence,
\begin{equation}\label{eqn:uniform_integrability_q}
 \sup_{x \in L,\ T \ge 0} \ee^x \{ |q(X_T)|^{1+\alpha} \} =: \Gamma < \infty.
\end{equation}
This estimate will play a crucial role in showing that $w_T$ converges to $w$ uniformly on $L$. Recall \eqref{eqn:diff_w_wT}:
\[
w(x) - w_T(x) \le \sup_{\tau \in \laa(x)} \ee^x \Big\{ \int_T^{\tau \vee T} f(X_s) ds \Big\} + 2 \|g\| \frac{M(x)}{T}.
\]
Since $M(x)$ is bounded on $L$ the second term converges uniformly to $0$. Take a stopping time $\tau$ with $\ee^x \{ \tau \} \le M(x)$:
\begin{align*}
\ee^x \Big\{ \int_T^{\tau \vee T} f(X_s) ds \Big\}
&=
\ee^x \big\{ \mu(f) (\tau \vee T - T) + q(X_T) - q(X_{\tau \vee T}) \big\}\\
&\le
\ee^x \{ \ind{\tau > T} q(X_T) \} - A\, \prob^x \{ \tau > T\},
\end{align*}
where $A$ is a lower bound for $q$ which can always be taken negative. By H\"older inequality, for any $x \in L$
\[
\ee^x \{ \ind{\tau > T} q(X_T) \} \le \big( \prob^x \{ \tau > T \} \big)^{1/q} \big( \ee^x \{ |q(X_T)|^{1+\alpha} \}\big)^{1/(1+\alpha)}
\le
\Big(\frac{M(x)}{T}\Big)^{1/q} \Gamma^{1/(1+\alpha)},
\]
where $q = (1+\alpha)/\alpha$ is adjoint to $1+\alpha$. Summarising,
\[
w(x) - w_T(x) \le \Big(\frac{M(x)}{T}\Big)^{1/q} \Gamma^{1/(1+\alpha)} + (2 \|g\| - A) \frac{M(x)}{T},
\]
which proves uniform convergence to $0$ on $L$ when $T \to \infty$.

\subsubsection{Relaxation of the lower bound on $f$}

The situation when $q$ is bounded from below can be interpreted as a guarantee that long waiting is not penalised severely. But, intuitively, harsh penalisation should make the problem easier as it would provide incentives to stop earlier and, therefore, prevent the explosion in the functional. This is indeed the case as we shall show in the theorem below.
\begin{thm}\label{thm:relax_B2}
Assume \ref{ass:speed}-\ref{ass:integrability} and one of the following conditions
\begin{itemize}
\item the process $(X_t)$ is strongly Feller, or
\item assumption \ref{ass:uniform_int}.
\end{itemize}
Then the value function $w$ is continuous.
\end{thm}
\begin{proof}
Assume first that there is a continuous bounded function $\bar f$ such that $\bar f \ge f$, $\mu(\bar f) < 0$ and the set $\{x \in E:\ \bar f(x) \le \mu(\bar f)\}$ is compact. We will relax it later. Let $z_n(x) = 1 - \rho(x, B_n) \wedge 1$, where $\rho(x, B_n)$ is the distance of $x$ from the ball of radius $n$. Put
\[
f_n(x) = \bar f(x) - z_n(x) \big( \bar f(x) - f(x) \big).
\]
Clearly, $f_n \equiv \bar f$ on $B_{n+1}^c$ and $f_n \le \bar f$. Hence, $\mu(f_n) \le \mu(\bar f)$,
\[
\{x \in E:\ f_n(x) \le \mu(f_n) \} \subset B_{n+1} \cup \{x \in E:\ \bar f(x) \le \mu(\bar f)\},
\]
and assumption \ref{ass:undis_large_dev} holds for $f_n$ by Lemma \ref{lem:integrability} and \ref{lem:f_C0}.
Denote by $w_n$ the value function corresponding to $f_n$, i.e.,
\[
w_n(x) = \sup_{\tau} \limsup_{T \to \infty} \ee^x \Big\{ \int_0^{\tau \wedge T} f_n(X_s) ds + g(X_{\tau \wedge T}) \Big\}.
\]
By Theorem \ref{thm:optimal_time_bounded_q} and discussion in the previous subsections $w_n$ is continuous and there is an optimal stopping time $\tau^*_n = \inf\{ t \ge 0:\ g(X_t) \ge w_n(X_t)\}$.
Since $f_n$ are decreasing in $n$ then $w_n$ are decreasing in $n$ and so do $\tau^*_n$. Hence, for any $T > 0$
\begin{align*}
&\ee^x \Big\{ \int_0^{\tau^*_n} f_n(X_s) ds + g(X_{\tau^*_n}) \Big\} - \ee^x \Big\{ \int_0^{\tau^*_n \wedge T} f_n(X_s) ds + g(X_{\tau^*_n \wedge T}) \Big\}\\
&\hspace{150pt}\le \|f_n\| \ee^x \{ (\tau^*_n - T)^+ \} + 2 \|g\| \prob^x \{ \tau^*_n > T \}\\[5pt]
&\hspace{150pt}\le (\|f\| + \|\bar f\|) \ee^x \{ (\tau^*_1 - T)^+ \} + 2 \|g\| \prob^x \{ \tau^*_1 > T \} =: H(T).
\end{align*}
This implies that $w_n(x) - w_{n,T}(x) \le H(T)$, where $w_{n,T}$ stands for the value function when stopping times are bounded by $T$. Notice that due to the integrability of $\tau^*_1$, we have $\lim_{T \to \infty} H(T) = 0$. Hence,
\begin{multline*}
w_n(x) - w(x)
= \big(w_n(x) - w_{n,T}(x) \big) + \big(w_{n,T}(x) - w_T(x) \big) + \big(w_T(x) - w(x)\big)\\
\le H(T) + \big(w_{n,T}(x) - w_T(x) \big)+ 0.
\end{multline*}
Pointwise convergence of $w_n$ to $w$ is established once we show that $w_{n,T}(x) - w_T(x)$ can be made arbitrarily small for every fixed $T$. By Lemma \ref{lem:compactness_of_Feller} for each $\ve > 0$ and a compact set $L \subset E$ there is  $N$ such that $\prob^x \big\{\forall t \in [0, T]\ X_t \in B_n \big\} \ge 1- \ve$ for $n \ge N$ and $x \in L$, i.e., the process stays in the ball $B_n$ over time $[0, T]$ with the probability at least $1-\ve$. Using the fact that $f_n$ coincides with $f$ on $B_n$, we have for every $n > N$
\[
0 \le w_{n, T}(x) - w_T(x) \le T \|f_n - f\|\, \prob^x \big\{\exists t \in [0, T]\ X_t \notin B_n \big\} \le T \big(\|f\| + \|\bar f\| \big)\ve.
\]
Summarising, for every $T > 0$, $\ve > 0$ and a compact set $L$ there is $N$ such that for all $n \ge N$ and $x \in L$ we have $w_n(x) - w(x) \le H(T) + T \big(\|f\| + \|\bar f\| \big)\ve$. This proves convergence of $w_n$ to $w$ uniform in $L$. Hence, the continuity of $w_n$ implies the continuity of $w$.

Take now an arbitrary continuous bounded function $f$ with $\mu(f) < 0$. Let $N$ be such that $\int_{B_N^c} |f(x)| \mu(dx) < -\mu(f)/4$. Define $\hat f(x) = z_N(x) f(x)$ and $\bar f = f \vee \hat f$. Then
\[
\mu(\bar f) \le \int_{B_N} f(x) \mu(dx) + \int_{B^c_N} |f(x)| \mu(dx) \le \mu(f) - \mu(f) / 4 - \mu(f) / 4 = \mu(f) / 2.
\]
Hence, $\mu(\bar f) < 0$. Moreover, $\bar f (x) \ge 0$ for $x \in B_{N+1}^c$, so the set $\{x\in E:\ \bar f(x) \le \mu(\bar f) \}$ is contained in $B_{N+1}$ and compact. The function $\bar f$ satisfies the conditions in the first part of the proof.
\end{proof}

\begin{remark}
The function $\bar f$ in the above proof satisfies conditions of Lemma \ref{lem:f_C0}. Hence, assumption \ref{ass:undis_large_dev} holds for $\bar f$. As it has been said earlier (see Remark \ref{rem:B1}), it also holds for any function dominated by $\bar f$, in particular, for $f$. Theorem \ref{thm:optimal_time_bounded_q} yields that $\tau^* = \inf\{ t \ge 0:\ g(X_t) \ge w(X_t) \}$ is optimal and has the expectation bounded by $M(x)$. Function $M(x)$ is bounded on compact sets because it is so for $\bar f$.
\end{remark}

\section{Functional with a general discount rate}\label{sec:general}
In the previous section we studied an optimal stopping problem without discounting. Here, we solve the problem with an arbitrary non-negative discount rate $r(x)$. The development will follow a similar line of thought as before but due to the presence of discounting many steps are more complicated.

For the sake of compactness of notation, we shall denote
\[
\alpha_t = \int_0^t r(X_s) ds.
\]
The optimal stopping problem \eqref{eqn:general} takes then the form
\begin{equation}\label{eqn:general_disc}
w(x) = \sup_{\tau} \limsup_{T \to \infty} \ee^x \Big\{ \int_0^{\tau \wedge T}  e^{-\alpha_s} f(X_s) ds  + e^{-\alpha_{\tau \wedge T}} g(X_{\tau \wedge T}) \Big\}.
\end{equation}
Standing assumptions for this section are \ref{ass:weak_feller} (weak Feller property), \ref{ass:continuity} (the continuity of $f, g$) and
\begin{assumption}
\item[(E1)]\label{ass:large_dev} For any $x \in E$, there is $d(x) < 0$ such that
\[
\gamma_r (x) = \sup_{\tau} \limsup_{T \to \infty} \ee^x \Big\{ \int_0^{\tau \wedge T} e^{-\alpha_s} \big(f(X_s) - d(x)\big) ds \Big\} < \infty,
\]
\item [(E2)]\label{ass:positive_discount} Function $r$ is continuous bounded and $r(x) \ge 0$ for all $x \in E$.
\end{assumption}
Similarly as in the undiscounted case, when assumption \ref{ass:large_dev} holds for a function $f$ it also holds for any function $f' \le f$ with the same function $d$.

\begin{lemma}\label{lem:general_disc_bound}
For any $x \in E$ and any integrable $\ve$-optimal stopping time $\sigma$ we have
\[
\ee^x \Big\{ \int_0^\sigma e^{-\alpha_s} ds \Big\} \le \frac{\gamma_r(x) + 2 \|g\| + \ve}{-d(x)}.
\]
The optimisation in \eqref{eqn:general_disc} can be constrained to stopping times $\sigma$ with
\[
\ee^x \Big\{ \int_0^\sigma e^{-\alpha_s} ds \Big\} \le M(x) := \frac{\gamma_r(x) + 2 \|g\| + 1}{-d(x)}
\]
and $w(x) \le \|f\| M(x) + \|g\| < \infty$.
\end{lemma}
The proof is similar to that of Lemma \ref{lem:bound_for_stop}.

Define a value function for the stopping problem with a finite horizon:
\[
w_T(x) = \sup_{\tau \le T} \ee^x \Big\{ \int_0^{\tau} e^{-\alpha_s} f(X_s)ds + e^{-\alpha_\tau} g(X_{\tau})\Big\}.
\]
\begin{lemma}\label{lem:disc_semicont}
Functions $w_T$ are continuous and the value function $w$ is lower semicontinuous.
\end{lemma}
\begin{proof}
The discounted semigroup $P^r_t \phi(x) = \ee^x \{ e^{-\int_0^t r(X_u) du} \phi(X_t) \}$ maps continuous bounded functions into continuous bounded functions, \cite[Chapter II, Section 5, Lemma 4]{Gikhman1975}. This implies the continuity of $w_T$.  As $w$ can be approximated pointwise from below by continuous functions $w_T$ it is itself lower semicontinuous.
\end{proof}
Hence, the random variable (using the convention that $\inf \emptyset = \infty$)
\[
\tau_\ve = \inf \{t \ge 0:\ g(X_t) \ge w(X_t) - \ve \}
\]
is a stopping time. On the set $\{\tau_\ve < \infty\}$, due to the continuity of $g$ and lower semicontinuity of $w$ as well as the right-continuity of $(X_t)$ we have
\[
g(X_{\tau_\ve}) \ge w(X_{\tau_\ve}) - \ve.
\]
It is not possible to determine without further assumptions whether $\tau_\ve$ admits only finite values. The following example explains why.
\begin{example}
Consider an optimal stopping problem $w(x) = \sup_{\tau} \limsup_{T \to \infty} \ee^x \{ e^{-(\tau \wedge T)r} g (X_{\tau \wedge T}) \}$ for $r >0$ and $g \equiv -1$. It is easy to see that there is a unique optimal stopping time $\tau^* = \infty$ and the value function $w(x) = 0$. However, for any $\ve \in (0,1)$ the set $\{ t:\ g(X_t) \ge w(X_t) - \ve \}$ is empty, so $\tau_\ve = \infty$.
\end{example}

This example suggests that although it cannot be guaranteed that $\tau_\ve$ is finite-valued it can still be $\ve$-optimal for $w$. This property is shared by the class of optimal stopping problems discussed in this section. The following technical lemmas provide us with tools required to demonstrate it. Proofs that are very similar to their counterparts in Section \ref{sec:no_disc} are omitted.

\begin{lemma}
For every $x \in E$, there exists a non-decreasing sequence $\sigma_m$ of bounded $\frac1m$-optimal stopping times for $w(x)$.
\end{lemma}
In what follows we shall refer by $(\sigma_m)$ to the sequence of stopping times from the above lemma.

\begin{lemma}\label{lem:gen_disc_snell}\label{lem:general_super_mart}
Assume only \ref{ass:positive_discount}. Then:
\begin{itemize}
\item The process $Z_t := \int_0^t e^{-\alpha_s} f(X_s) ds + e^{-\alpha_t} w(X_t)$ is a right-continuous $\prob^x$-supermar\-tingale for any $x \in E$.
\item For a bounded stopping time $\sigma$ and an arbitrary stopping time $\tau$
\begin{multline}\label{eqn:gen_disc_bellman_ineq}
\ee^{x} \Big \{ \int_0^\sigma e^{-\alpha_s} f(X_s) ds + e^{-\alpha_\sigma} g(X_\sigma) \Big\}\\
\le \ee^{x} \Big\{ \int_0^{\sigma\wedge \tau} e^{-\alpha_s} f(X_s) ds + \ind{\sigma < \tau} e^{-\alpha_\sigma} g(X_{\sigma}) + \ind{\sigma \ge \tau} e^{-\alpha_\tau} w(X_{\tau}) \Big\}.
\end{multline}
\item For a bounded stopping time $\sigma$
\begin{equation}\label{eqn:gen_disc_bellman_ineq2}
\ee^{x} \Big \{ \int_0^\sigma e^{-\alpha_s} f(X_s) ds + e^{-\alpha_\sigma} w(X_\sigma) \Big\} \le w(x).
\end{equation}
\end{itemize}
\end{lemma}
\begin{proof}
First two statements are proved in a similar way as Lemma \ref{lem:undisc_snell}. Inequality \eqref{eqn:gen_disc_bellman_ineq2} follows from \eqref{eqn:gen_disc_bellman_ineq} by taking $\tau \equiv 0$.
\end{proof}

\begin{lemma}\label{lem:gen_disc_tech}
$\ $
\begin{enumerate}
\item $\displaystyle \ee^x \big\{ \ind{\sigma_m < \tau_\ve} e^{-\alpha_{\sigma_m}}\big\} \le \frac{1}{m \ve}.$
\item $\displaystyle \lim_{m \to \infty} \ee^x \big\{ \ind{\tau_\ve < \infty} \ind{\sigma_m \ge \tau_\ve} e^{-\alpha_{\tau_\ve}} \big\} = \ee^x \big\{ \ind{\tau_\ve < \infty } e^{-\alpha_{\tau_\ve}} \big\}.$
\end{enumerate}
\end{lemma}
\begin{proof}
We follow similar lines as in the proof of Lemma \ref{lem:undisc_tech}. Apply \eqref{eqn:gen_disc_bellman_ineq} to $\sigma_m$ and $\tau_\ve$
\begin{align*}
w(x) - \frac1m
&\le \ee^{x} \Big\{ \int_0^{\sigma_m\wedge \tau_\ve} e^{-\alpha_s} f(X_s) ds + \ind{\sigma_m < \tau_\ve} e^{-\alpha_{\sigma_m}} g(X_{\sigma_m}) + \ind{\sigma_m \ge \tau_\ve} e^{-\alpha_{\tau_\ve}} w(X_{\tau_\ve}) \Big\}\\
&\le  \ee^{x} \Big\{ \int_0^{\sigma_m\wedge \tau_\ve} e^{-\alpha_s} f(X_s) ds + \ind{\sigma_m < \tau_\ve} e^{-\alpha_{\sigma_m}} \big( w(X_{\sigma_m}) - \ve \big) + \ind{\sigma_m \ge \tau_\ve} e^{-\alpha_{\tau_\ve}} w(X_{\tau_\ve})\Big\}\\
&\le  \ee^{x} \Big\{ \int_0^{\sigma_m\wedge \tau_\ve} e^{-\alpha_s} f(X_s) ds + e^{-\alpha_{\sigma_m \wedge \tau_\ve}} w(X_{\sigma_m \wedge \tau_\ve}) \Big\} - \ve \ee^x \big\{\ind{\sigma_m < \tau_\ve} e^{-\alpha_{\sigma_m}} \big\} \\
&\le w(x) - \ve \ee^x \big\{\ind{\sigma_m < \tau_\ve} e^{-\alpha_{\sigma_m}} \big\},
\end{align*}
where the second inequality follows from the observation that $g(X_{\sigma_m}) < w(X_{\sigma_m}) - \ve$ on $\{ \sigma_m < \tau_\ve \}$ and the last inequality is the result of \eqref{eqn:gen_disc_bellman_ineq2}. This proves the first statement of the lemma.

The second statement is a consequence of the first one. Write
\[
\ee^x \big\{ \ind{\tau_\ve < \infty } e^{-\alpha_{\tau_\ve}} \big\} = \ee^x \big\{ \ind{\tau_\ve < \infty } \ind{\sigma_m < \tau_\ve} e^{-\alpha_{\tau_\ve}} \big\} + \ee^x \big\{ \ind{\tau_\ve < \infty } \ind{\sigma_m \ge \tau_\ve} e^{-\alpha_{\tau_\ve}} \big\}.
\]
Recall that $\alpha_s$ is non-decreasing:
\[
\ee^x \big\{ \ind{\tau_\ve < \infty } \ind{\sigma_m < \tau_\ve} e^{-\alpha_{\tau_\ve}} \big\}
\le \ee^x \big\{ \ind{\tau_\ve < \infty } \ind{\sigma_m < \tau_\ve} e^{-\alpha_{\sigma_m}} \big\} \to 0
\]
as $m \to \infty$ by the first assertion of the lemma.
\end{proof}

The following lemma unveils an important relation between $\tau_\ve$ and $\sigma_m$.

\begin{lemma}\label{lem:disc_relation_sigma_m}
Under the standing assumptions of this section
\[
\lim_{m \to \infty} (\sigma_m \wedge \tau_\ve) = \tau_\ve.
\]
\end{lemma}
\begin{proof}
Stopping times $\sigma_m$ are non-decreasing, hence the limit $\sigma_\infty = \lim_{m \to \infty} \sigma_m$ exists (and is possibly infinite). From Lemma \ref{lem:gen_disc_tech}, by the dominated convergence theorem
\[
0 = \ee^x \big\{ \lim_{m \to \infty} \ind{\sigma_m < \tau_\ve} e^{-\alpha_{\sigma_m}}\big\} \ge \ee^x \big\{ \lim_{m \to \infty} \ind{\sigma_m < \tau_\ve} e^{-\|r\| {\sigma_m}}\big\}.
\]
Hence, on $\{ \sigma_\infty < \infty \}$ we have $\lim_{m \to \infty} \ind{\sigma_m < \tau_\ve} = 0$, which implies $\sigma_\infty \ge \tau_\ve$. Obviously on $\{\sigma_\infty = \infty \}$ we have $\sigma_\infty \ge \tau_\ve$.
\end{proof}

We are now in a position to prove the main result of this section that $\tau_\ve$ is an $\ve$-optimal stopping time for $w$. This proof does not require $\tau_\ve$ to be finite.

\begin{thm}\label{thm:general_disc_optimality}
Under the standing assumptions of this section the stopping time $\tau_\ve$ is $\ve$-optimal for $w$.
\end{thm}
\begin{proof}
Apply \eqref{eqn:gen_disc_bellman_ineq} to $\sigma_m$ and $\tau_\ve$
\begin{align}
w(x) - \frac1m
&\le \ee^{x} \Big \{ \int_0^{\sigma_m} e^{-\alpha_s} f(X_s) ds + e^{-\alpha_{\sigma_m}} g(X_{\sigma_m}) \Big\}\notag\\
&\le \ee^{x} \Big\{ \int_0^{\sigma_m\wedge \tau_\ve} e^{-\alpha_s} f(X_s) ds + \ind{\sigma_m < \tau_\ve} e^{-\alpha_{\sigma_m}} g(X_{\sigma_m}) + \ind{\sigma_m \ge \tau_\ve} e^{-\alpha_{\tau_\ve}} w(X_{\tau_\ve}) \Big\}.\label{eqn:ineq_tau_ve}
\end{align}
We now want to find out how the right-hand side of this inequality looks like when $m \to \infty$. By assertion 1 of Lemma \ref{lem:gen_disc_tech},
\[
\lim_{m \to \infty} \ee^x \big\{ \ind{\sigma_m < \tau_\ve} e^{-\alpha_{\sigma_m}} g(X_{\sigma_m}) \big\} = 0.
\]
On $\{ \sigma_m \ge \tau_\ve\}$ we have $\tau_\ve < \infty$, so $w(X_{\tau_\ve}) \le g(X_{\tau_\ve}) + \ve$. Hence, for each $m$
\begin{multline*}
\ee^x \big\{ \ind{\sigma_m \ge \tau_\ve} e^{-\alpha_{\tau_\ve}} w(X_{\tau_\ve}) \big\}
\le
\ee^x \big\{ \ind{\sigma_m \ge \tau_\ve} e^{-\alpha_{\tau_\ve}} g(X_{\tau_\ve}) \big\} + \ve\\
=
\ee^x \big\{ \ind{\tau_\ve < \infty} \ind{\sigma_m \ge \tau_\ve} e^{-\alpha_{\tau_\ve}} g(X_{\tau_\ve}) \big\}
+
\ee^x \big\{ \ind{\tau_\ve = \infty} \ind{\sigma_m \ge \tau_\ve} e^{-\alpha_{\tau_\ve}} g(X_{\tau_\ve}) \big\} + \ve.
\end{multline*}
Clearly, the second expectation is $0$, while the first one converges to
\[
\ee^x \big\{ \ind{\tau_\ve < \infty} e^{-\alpha_{\tau_\ve}} g(X_{\tau_\ve}) \big\}
\]
by the second assertion of Lemma \ref{lem:gen_disc_tech}. A proof of convergence of the first term on the right-hand side of \eqref{eqn:ineq_tau_ve} is slightly more involved. Notice that
\[
\Big| \int_0^{\sigma_m \wedge \tau_\ve} e^{-\alpha_s} f(X_s) ds \Big| \le \|f\| \int_0^{\sigma_m \wedge \tau_\ve} e^{-\alpha_s} ds.
\]
The right-hand side is increasing in $m$. Hence, it is bounded by a random variable $\|f\| \int_0^{\sigma_\infty \wedge \tau_\ve} e^{-\alpha_s} ds=\|f\| \int_0^{\tau_\ve} e^{-\alpha_s} ds$ whose expectation does not exceed $\|f\| M(x)$. This allows us to use the dominated convergence theorem from which we obtain
\[
\lim_{m \to \infty} \ee^{x} \Big\{ \int_0^{\sigma_m\wedge \tau_\ve} e^{-\alpha_s} f(X_s) ds \Big\} =
\ee^{x} \Big\{ \int_0^{\tau_\ve} e^{-\alpha_s} f(X_s) ds \Big\}.
\]
Since $\int_0^{\tau_\ve} e^{-\alpha_s} ds\le M(x)$ we also have
\[
\ee^{x} \Big\{ \int_0^{\tau_\ve} e^{-\alpha_s} f(X_s) ds \Big\} = \limsup_{T \to \infty} \ee^{x} \Big\{ \int_0^{\tau_\ve \wedge T} e^{-\alpha_s} f(X_s) ds \Big\}.
\]
Combining the above results gives
\[
w(x) \le \lim_{T \to \infty} \ee^{x} \Big \{ \int_0^{\tau_\ve \wedge T} e^{-\alpha_s} f(X_s) ds\Big\} + \ee^x \Big\{ \ind{\tau_\ve < \infty} e^{-\alpha_{\tau_\ve}} g(X_{\tau_\ve}) \Big\} + \ve.
\]
It remains to prove that
\begin{equation}\label{eqn:gen_disc_aux1}
\limsup_{T \to \infty} \ee^x \Big\{ \ind{\tau_\ve = \infty} e^{-\alpha_{T}} g(X_{T}) \Big\} = 0.
\end{equation}
By the first assertion of Lemma \ref{lem:gen_disc_tech}, we have
\[
\ee^x \big\{ \ind{\tau_\ve = \infty} e^{-\sigma_m} \big\} \le \frac{1}{m \ve}.
\]
By monotonicity of $t \mapsto \alpha_t$ and the dominated convergence theorem, this implies
\[
\frac{1}{m \ve} \ge \lim_{T \to \infty} \ee^x \big\{ \ind{\tau_\ve = \infty} e^{-\alpha_{\sigma_m \wedge T}} \big\} \ge \lim_{T \to \infty} \ee^x \big\{ \ind{\tau_\ve = \infty} e^{-\alpha_T} \big\}.
\]
This is true for an arbitrary $m$, so the limit on the right-hand side is in fact $0$. Boundedness of $g$ then yields \eqref{eqn:gen_disc_aux1}, which implies
\[
w(x) \le \limsup_{T \to \infty} \ee^{x} \Big \{ \int_0^{\tau_\ve \wedge T} e^{-\alpha_s} f(X_s) ds +  e^{-\alpha_{\tau_\ve \wedge T}} g(X_{\tau_\ve \wedge T}) \Big\} + \ve.
\]
Hence $\tau_\ve$ is $\ve$-optimal.
\end{proof}

We have shown that an $\ve$-optimal stopping time exists in a standard form. Without further assumptions we are not able to construct an optimal stopping time. This should not be surprising as this is common for infinite horizon stopping problems even with a constant discount rate. Indeed, our main tool in the undiscounted case is the integral part of the functional which is interpreted as a penalty for a long wait before stopping -- an infinite-valued stopping time drives the functional to $-\infty$. When discounting is applied, the integral part of the functional may remain bounded whatever the stopping time is applied; this is clearly the case when $\inf_x r(x) > 0$. However, when the discounting is limited, we are able to say more about an optimal stopping time and about properties of the value function.
\begin{assumption}
\item[(D3)]\label{ass:small_discount} The random variable $R := \lim_{t \to \infty} \alpha_t$ is $\prob^x$-a.s. finite.
\end{assumption}
The random variable $R$ is well-defined because $r(x) \ge 0$ so $\alpha_t$ is non-decreasing. It is satisfied if the discount rate $r(X_t)$ decreases quickly with time. In particular, a constant discount rate does not satisfy \ref{ass:small_discount}.
\begin{assumption}
\item[(D3')]\label{ass:small_discount_deterministic} There is $\delta > 0$ such that $\alpha_t \le -\log(\delta)$, $\prob^x$-a.s. for all $t \ge 0$.
\end{assumption}
The simplest example of a stopping problem satisfying \ref{ass:small_discount_deterministic} is $r \equiv 0$, i.e., the problem studied in Section \ref{sec:no_disc}. Notice that assumption \ref{ass:small_discount_deterministic} is a particular case of assumption \ref{ass:small_discount} with $R \le - \log(\delta)$.

Let
\[
\tau^* = \inf \{ t \ge 0:\ g(X_t) \ge w(X_t)\}.
\]
On the set $\{\tau^* < \infty\}$, due to the continuity of $g$ and lower semicontinuity of $w$ as well as the right-continuity of $(X_t)$ we have
\[
g(X_{\tau^*}) \ge w(X_{\tau^*}).
\]
We shall use this property in the proof of the following theorem.
\begin{thm}\label{thm:general_optimality}
Assume \ref{ass:small_discount} in addition to the standing assumptions of this section. The stopping time $\tau^*$ is finite $\prob^x$-a.s. and optimal for $w(x)$. Moreover, $\ee^x \{ \tau^* e^{-R} \} \le M(x)$, where $M(x)$ is defined in Lemma \ref{lem:general_disc_bound}. Under \ref{ass:small_discount_deterministic} we additionally have $\ee^x \{ \tau_\ve \} \le M(x)/\delta$ and $\ee^x \{ \tau^* \} \le M(x)/\delta$.
\end{thm}
\begin{proof}
We have
\[
\ee^x \{ \sigma_m e^{-R} \} \le \ee^x \Big\{ \int_0^{\sigma_m} e^{-\alpha_s} ds \Big\} \le M(x),
\]
where the first inequality follows from the monotonicity of $\alpha_s$ and assumption \ref{ass:small_discount} while the second inequality is from Lemma \ref{lem:general_disc_bound}. By the first assertion of Lemma \ref{lem:gen_disc_tech} we obtain
\[
\ee^x \{ \ind{\sigma_m < \tau_\ve} e^{-R} \} \le \frac{1}{m \ve}.
\]
The factor $e^{-R}$ can be interpreted as an unnormalised Radon-Nikodym density. Using this density to change the probability measure we end up in the setting of  Lemma \ref{lem:bound_expectation}. Hence, $\ee^x \{ \tau_\ve e^{-R} \} \le M(x)$. In Theorem \ref{thm:general_disc_optimality} we showed that $\tau_\ve$ is $\ve$-optimal.  Put $\tau_0 = \lim_{\ve \to 0} \tau_\ve$. Since this is an increasing sequence of random variables with the weighted expectation bounded by $M(x)$ then $\tau_0$ is well defined and $\ee^x \{ \tau_0 e^{-R} \} \le M(x)$. Since $R$ is finite-valued, this implies that $\tau_0 < \infty$ $\prob^x$-a.s. We also have that  $g(X_{\tau_\ve}) \ge w(X_{\tau_\ve}) - \ve$. By the quasi left-continuity of $(X_t)$, taking the limit $\ve \to 0$ we obtain
\[
g(X_{\tau_0}) \ge w(X_{\tau_0}).
\]
So $\tau_0 \ge \tau^*$. This means that $\tau^*$ is finite-valued and $\ee^x \{ \tau^* e^{-R} \} \le M(x)$. Clearly, $\tau_\ve \le \tau^*$, so $\tau_0 = \tau^*$. The optimality of $\tau^*$ follows immediately by letting $\ve \to 0$ in
\[
w(x)-\ve \leq \ee^x \Big\{\int_0^{\tau_\ve} e^{-\alpha_s} f(X_s)ds + e^{-\alpha_{\tau_\ve}}g(X_{\tau_\ve})\Big\}.
\]

Under assumption \ref{ass:small_discount_deterministic} we have $e^{-R} \ge \delta$. Then $\ee^x \{ \tau_\ve \, \delta\} \le \ee^x \{ \tau_\ve e^{-R} \}  \le M(x)$ and, analogously, $\ee^x \{ \tau^*\, \delta \} \le M(x)$. Recalling that $\delta$ is a constant completes the proof.
\end{proof}

\subsection{Sufficient conditions for \ref{ass:large_dev}}
We shall present two conditions that imply \ref{ass:large_dev}. They generalise corresponding conditions from Section \ref{sec:no_disc}.

Define a potential
\[
q_r(x) = \limsup_{T \to \infty} \ee^x \Big\{ \int_0^T e^{-\alpha_s} \big(f(X_s) - \mu(f)\big) ds \Big\}.
\]
and assume:
\begin{assumption}
\item[(C1')] $q_r(x)$ is continuous and bounded from below, \label{ass:B1_prime}
\item[(C2')] for any bounded stopping time $\sigma$
\[
q_r(x)  = \ee^x \Big\{ \int_0^{\sigma} e^{-\alpha_s} \big(f(X_s) - \mu(f) \big) ds + e^{-\alpha_\sigma} q_r(X_\sigma)\Big\}.
\] \label{ass:B2_prime}
\end{assumption}
These assumptions collapse to their counterparts in Section \ref{sec:no_disc} when $r \equiv 0$, i.e., $\alpha_t = 0$.
Sufficient conditions for \ref{ass:B1_prime} under uniform (geometric) ergodicity of $(X_t)$ can be found in \cite[Theorem 4]{Stettner1989a} (see also \cite{KonMeyn}). An easy adaptation of Lemma \ref{lem:integrability} and \ref{lem:f_C0} shows that conditions \ref{ass:B1_prime}-\ref{ass:B2_prime} are satisfied when $\mu(f) < 0$, the set $L = \{ x \in E:\ f(x) \le \mu(f) \}$ is compact and assumptions \ref{ass:speed}-\ref{ass:integrability} are fulfilled.

The statement and the proof of the following lemma resembles closely Lemma \ref{lem:gamma_from_q}.
\begin{lemma}
Assume \ref{ass:B1_prime}-\ref{ass:B2_prime} and \ref{ass:B3}. Assumption \ref{ass:large_dev} is satisfied with any $d(x) \in (\mu(f), 0)$. Moreover, $\gamma_r(x) \le q_r(x) - A$, where $A = \min\big(0, \inf_y q_r(y) \big)$.
\end{lemma}

The second criterion for \ref{ass:large_dev} is provided by an upper bound for large deviations of the discounted empirical process (cf. assumption \ref{ass:L}).
\begin{lemma}\label{lem:gen_ldev_bound}
Assume
\begin{assumption}
\item[(L$_r$)] For any $\delta, \ve>0$ and a compact set $K \subset E$ there is $N>0$ and $p>0$ such that for $n\ge N$ and $x \in K$ we have
\[
\prob^x\Big\{\Big|\frac1{\int_0^{n \delta} e^{-\alpha_s} ds} \int_0^{n \delta} e^{-\alpha_s} f(X_s) ds - \mu(f)\Big|> \ve \Big\} \le e^{-p (n \delta)}.
\]\label{ass:L_r}
\end{assumption}
Then \ref{ass:large_dev} holds and $M(x)$ defined in Lemma \ref{lem:general_disc_bound} is bounded on compact sets.
\end{lemma}
\begin{proof}
This proof is similar to the proof of Lemma \ref{lem:ldev_bound}. Choose $0 < \ve < -\mu(f)$ and fix $x \in E$. Let $A_t = \big\{ \sup_{u \ge t} \big| \big(\int_0^u e^{-\alpha_s} ds\big)^{-1} \int_0^u e^{-\alpha_s} f(X_s) ds -\mu(f)\big| >\ve \big\}$. Similarly as in Lemma \ref{lem:supremum_large_dev} there is $N > 0$ such that for $t \ge N$ we have $\prob^x \{ A_t \} \le Ce^{-\rho t}$. Then for any integrable stopping time $\tau$ we have
\begin{align*}
\ee^x\left\{\int_0^\tau e^{-\alpha_s} (f(x_s)-\mu(f)) ds\right\}
&= \sum_{i=0}^\infty \ee^x \left\{1_{i\leq \tau <i+1} \int_0^\tau e^{-\alpha_s} (f(x_s)-\mu(f))ds\right\}\\
&\le
\ve \ee^x\Big\{ \int_0^\tau e^{-\alpha_s} ds \Big\} +2\|f\| +\sum_{i=1}^\infty 2\|f\| (i+1) \prob^x\{A_i\}\\
&\le
\ve \ee^x\Big\{ \int_0^\tau e^{-\alpha_s} ds \Big\} +2\|f\|+ \sum_{i=1}^\infty 2\|f\| (i+1) Ce^{-i\rho}\\
&=\ve \ee^x\Big\{ \int_0^\tau e^{-\alpha_s} ds \Big\} + \bar{C}.
\end{align*}
This is equivalent to
\[
\ee^x\left\{\int_0^\tau e^{-\alpha_s} \big(f(x_s)-\mu(f)-\ve\big) ds\right\} \le \bar C.
\]
Hence, \ref{ass:large_dev} holds with $d(x) = \mu(f) + \ve$ which, by the choice of $\ve$, is strictly negative. Since $\bar C$ can be taken bounded on compact sets, we immediately obtain that $d$ can be taken constant on compact sets and then $M(x)$ is bounded on compact sets.
\end{proof}

\subsection{Continuity of the value function}
In three out of four cases studied in the theorem below, the continuity of $w$ results from the convergence of $w_T$ to $w$ uniformly on compact sets and proofs from Subsection \ref{subsec:continuity} apply. The proof when the process $(X_t)$ is strongly Feller retains its validity due to the boundedness of the discount rate $r$ (see assumption \ref{ass:positive_discount}).

\begin{thm}\label{thm:gen_continuity}
Assume \ref{ass:large_dev}-\ref{ass:positive_discount} and \ref{ass:weak_feller}. Either of the following conditions is sufficient for the continuity of the value function $w$:
\begin{enumerate}
\item Assumption \ref{ass:L_r} holds.
\item $f(x) \le 0$ for $x \in E$.
\item The process $(X_t)$ is strongly Feller and assumptions \ref{ass:B1_prime}-\ref{ass:B2_prime} hold.
\item Assumption \ref{ass:speed} and \ref{ass:uniform_int} hold.
\end{enumerate}
\end{thm}
\begin{proof}
The proof is an easy adaptation of the reasoning in Subsection \ref{subsec:continuity}.
\end{proof}

\section{Variational characterisation of the value function}\label{sec:VI}
Previous sections discuss properties of the value function and existence of optimal stopping times. This section provides a more explicit description of the value function as a solution to a variational inequality
\begin{equation}\label{eqn:VI}
\min \big( -\laa w + rw - f, w - g \big) = 0,
\end{equation}
where $\laa$ is a generator for the process $(X_t)$. In general, it is unlikely that the value function is in the domain $D_\laa$ of the generator, so this variational formulation cannot be interpreted in a classical sense. Instead, we shall show that the value function is a viscosity solution of \eqref{eqn:VI}. We shall employ a more stringent definition than commonly used, i.e., we shall use a larger class of test functions.
\begin{definition}
A continuous function $u$ is a \emph{viscosity subsolution} of \eqref{eqn:VI} if for each $x \in E$ and $\psi \in D_\laa$ such that $u(x) = \psi(x)$ and $\psi \ge u$ we have
\begin{equation}\label{eqn:VI_subsolution}
\min \big( -\laa \psi(x) + r(x) \psi(x) - f(x), u(x) - g(x) \big) \le 0.
\end{equation}

A continuous function $u$ is a \emph{viscosity supersolution} of \eqref{eqn:VI} if for each $x \in E$ and $\phi \in D_\laa$ such that $u(x) = \phi(x)$ and $\phi \le u$ we have
\begin{equation}\label{eqn:VI_supersolution}
\min \big( -\laa \phi(x) + r(x)\phi(x) - f(x), u(x) - g(x) \big) \ge 0.
\end{equation}

A continuous function $u$ is a \emph{viscosity solution} of \eqref{eqn:VI} if it is both super- and subsolution.
\end{definition}

In practice, the domain of a generator is rarely known. Instead, one considers a \emph{core}, a linear subspace of $D_\laa$ that defines the generator $A$ on $D_\laa$ uniquely via the closure of its graph. Test functions in the definition of the viscosity solutions are then restricted to belong to that core. For diffusions, it is common to consider $C^2_0$ functions or even $C^\infty_0$ functions. Our choice to use all functions from the domain of the generator as test functions in the definition of viscosity solutions gives a flexibility for the selection of a core in the subsequent quest to solve \eqref{eqn:VI} and to prove the uniqueness of solutions. This also comes at no additional cost.

Assume \ref{ass:weak_feller}-\ref{ass:continuity}, \ref{ass:large_dev}-\ref{ass:positive_discount} and \ref{ass:small_discount_deterministic}. They ensure that there is an optimal stopping time with a finite expectation. We also require that the value function $w$ is continuous. Theorem \ref{thm:gen_continuity} provides sufficient conditions.

The domain of a generator of a weakly Feller process is a dense subset of the space of continuous functions vanishing in infinity \cite[Theorem 17.4]{Kallenberg1997}. The range of the generator is within the same space -- we shall use this continuity property in the proof of the following theorem.

\begin{thm}\label{thm:viscosity_characterisation}
Under the standing assumptions of this section, the value function $w$ is a viscosity solution of \eqref{eqn:VI}.
\end{thm}
\begin{proof}
\emph{Supersolution property:} Take $x \in E$ and $\phi \in D_\laa$ as in the definition of supersolution. Obviously, $\phi(x) - g(x) = w(x) - g(x) \ge 0$.  By Lemma \ref{lem:general_super_mart} for any $t > 0$ we have
\[
w(x) \ge \ee^x \Big\{ \int_0^t e^{-\alpha_s} f(X_s) ds + e^{-\alpha_t} w(X_t) \Big\}.
\]
Thus
\[
\ee^x \big\{ e^{-\alpha_t} \phi(X_t) \big\} - \phi(x) + \ee^x \Big\{ \int_0^t e^{-\alpha_s} f(X_s) ds \Big\} \le 0.
\]
By Dynkin's formula
\[
\ee^x \Big\{ \int_0^t e^{-\alpha_s} \Big( \laa \phi(X_s) - r(X_s) \phi(X_s) + f(X_s) \Big) ds \Big\} \le 0.
\]
Dividing by $t$ both sides of the last inequality and letting $t \to 0$ we obtain
$\laa \phi(x) - r(x) \phi(x) + f(x) \le 0$.

\emph{Subsolution property:} Take $x \in E$ and $\psi \in D_\laa$ as in the definition of subsolution. If $w(x) = g(x)$ then \eqref{eqn:VI_subsolution} is trivially satisfied. Otherwise, $\tau^* > 0$ $\prob^x$-a.s. By Lemma \ref{lem:gen_disc_snell}, the process $Z_t := \int_0^t e^{-\alpha_s} f(X_s) ds + e^{-\alpha_t} w(X_t)$ is a right-continuous $\prob^x$-supermartingale. Since $\ee^x \{Z_{\tau^*} \} = Z_0$ by the optimality of $\tau^*$, the process $Z_{t \wedge \tau^*}$ is a  martingale and therefore
\[
\ee^x \Big\{ e^{-\alpha_{\tau^* \wedge t}} \psi (X_{\tau^* \wedge t}) - \psi(x) + \int_0^{\tau^* \wedge t} e^{-\alpha_s} f(X_s) ds \Big\}\ge 0.
\]
Dynkin's formula yields for any $t > 0$
\[
\ee^x \Big\{ \int_0^{\tau^* \wedge t} e^{-\alpha_s} \Big( \laa \psi(X_s) - r(X_s) \psi(X_s) + f(X_s) \Big) ds \Big\} \ge 0 .
\]
Dividing both sides of the last inequality by $t$ and letting $t \to 0$, taking into account that $\tau^* > 0$ $\prob^x$-a.s., we obtain that $\laa \psi(x) - r(x) \psi(x) + f(x) \ge 0$.
\end{proof}

Demonstration of uniqueness of viscosity solutions to the variational inequality \eqref{eqn:VI} is much more involved. We shall only conjecture here that when the process $X_t$ is a weakly Feller jump-diffusion on $\er^d$ and viscosity solutions are defined by a core of $C^2_0$ test function then one can obtain uniqueness of solutions within a family of functions with a polynomial growth. The sufficiency of this notion depends very much on whether the value function has a polynomial growth. This growth is related to the expectation of the optimal stopping time $\tau^*$. Under the setting of \ref{ass:speed}-\ref{ass:integrability}, this expectation is of the order  $O(K(x))$, where $K(x)$ is from assumption \ref{ass:speed}.

\section{Dichotomy of discounting}\label{sec:dichotomy}
In this section we drop the assumption $\mu(f) < 0$ and consider a general problem with a continuous bounded discount rate function $r$. We show that under certain assumptions (to be specified below) there are effectively two distinct regimes of the optimal stopping problem. When $\mu(r) > 0$, the stopping problem exhibits features of the classical optimal stopping problem with the discount rate separated from zero (cf. \cite{Robin1978, Zabczyk1984}). In particular, the value function is continuous (hence, finite) for any continuous bounded function $f$ and $w_T$ approximates $w$ uniformly on compact sets. On the other hand, when $\mu(r) = 0$ it behaves as if the discount rate was equal to $0$. Indeed, the discount rate is non-negative so the assumption that $\mu(r) = 0$ implies that $r \equiv 0$ on the support of the invariant measure $\mu$. It is common in many applications that the invariant measure is supported by the whole space $E$ so $\int_0^t r(s) ds = 0$, $\prob^x$-a.s. for all $t \ge 0$. In an unlikely case when the space contains a transient set ($\mu$ is not supported by the whole space) and the time until reaching the recurrent set is integrable, assumption \ref{ass:small_discount} is satisfied and we can employ Theorem \ref{thm:general_optimality}.

Consider now $\mu(r) > 0$ and assume that the following upper bound for large deviations of the empirical process holds:
\begin{assumption}
\item[(L$_d$)] For any $\delta, \ve>0$ and a compact set $K \subset E$ there is $N>0$ and $p>0$ such that for $n\ge N$ and $x \in K$ we have
\[
\prob^x\Big\{\Big|\frac1{n \delta} \int_0^{n \delta} r(X_s)ds - \mu(r)\Big|> \ve \Big\} \le e^{-p (n \delta)}.
\] \label{ass:L_d}
\end{assumption}
By an identical argument as in Lemma \ref{lem:supremum_large_dev} there is a constant $C > 0$ and $\rho > 0$ such that for $t \ge S := N \delta$
\[
\prob^x\Big\{ \sup_{s\geq t} \Big|\frac1s \int_0^s r(X_u)du-\mu(r)\Big| >\ve\Big\}\le Ce^{-\rho t}, \qquad \forall\ x \in K.
\]

Choose $0 < \ve < \mu(r)$ and let
\[
A_t = \Big\{ \sup_{s\geq t} \Big|\frac1s \int_0^s r(X_u)du-\mu(r)\Big| >\ve\Big\}.
\]
Consider a modified stopping problem
\[
\tl w (x) = \sup_{\tau} \ee^x \Big\{ \int_0^\tau e^{-(\alpha_s \vee \lambda s)} f(X_s) ds + e^{-(\alpha_\tau \vee \lambda \tau)} g(X_\tau) \Big\},
\]
where $\lambda = \mu(r) - \ve$. One can solve this problem in a standard way (cf. \cite{Robin1978}). Denoting by $\tl w_T$ the value function with stopping times bounded by $T$ we have
\[
0 \le \tl w(x) - \tl w_T(x) \le \frac{1}{\lambda} e^{-\lambda T} \|f\| + 2 e^{-\lambda T} \|g\|.
\]
Functions $\tl w_T$ are continuous \cite[Corollary 3.6]{Palczewski2008}. They approximate $\tl w$ uniformly on compact sets, which implies that $\tl w$ is continuous. $\ve$-optimal stopping times take a standard form, while optimal stopping times might not exist as in the case with a constant discount rate.

It is interesting to note that under \ref{ass:L_d} the difference $w - \tl w$ is bounded. Indeed, for $x \in K$
\begin{align*}
w(x) - \tl w(x)
&\le \sup_{\tau, \text{bounded}} \ee^x \Big\{ \int_0^\tau \big(e^{-\alpha_s} - e^{-(\alpha_s \vee \lambda s)}\big) f(X_s) ds + \big(e^{-\alpha_\tau} - e^{-(\alpha_\tau \vee \lambda \tau)} \big) g(X_\tau) \Big\}.
\end{align*}
Recall that $\prob^x \{A_t\} \le C e^{-\rho t}$ for $t \ge S$. Let $n = \lfloor S \rfloor + 1$. Since on $A_t^c$ there is $\alpha_s \ge \lambda s$ for $s \ge t$, we have
\begin{align*}
w(x) - \tl w(x) &\le \sup_{\tau, \text{bounded}} \ee^x \Big\{ \sum_{i=n}^\infty \ind{i \le \tau} \ind{A_i} \|f\|  \Big\} + n \|f\| + \|g\|\\
&\le \|f\| \sum_{i=n}^\infty \prob^x \{ A_i\} + n \|f\| + \|g\| \\
&\le \|f\| \Big(n + \frac{C}{\rho} e^{-\rho n} \Big) + \|g\|.
\end{align*}
Similarly, one can estimate the difference $\tl w(x) - w(x)$. This means that $w(x)$ is finite, so the integral term cannot explode. Moreover, this similarity between the two optimal stopping problem guides the following lemma.
\begin{lemma}
Assume \ref{ass:L_d} and $\mu(r) > 0$. The value function $w$ is continuous and approximated uniformly on compact sets by $w_T$.
\end{lemma}
\begin{proof}
Fix a compact set $K \subset E$ and $\ve < \mu(r)$. Take $T \ge S$, where $S$ is introduced below the large deviations assumption \ref{ass:L_d}. For any $x \in E$
\begin{align*}
w(x) - w_T(x)
&\le
\sup_{\tau, \text{bounded}} \ee^x \Big\{ \int_{\tau \wedge T}^\tau e^{-\alpha_s} f(X_s) ds + \ind{\tau > T} \big(e^{-\alpha_\tau} g(X_\tau) - e^{-\alpha_{\tau \wedge T}} g(X_{\tau \wedge T}) \big)\Big\}\\
&\le
\sup_{\tau, \text{bounded}} \ee^x \Big\{ \sum_{i=0}^\infty \ind{\tau \ge T + i} e^{-\alpha_{T+i}} \|f\|  + \ind{\tau > T} e^{-\alpha_T} 2\|g\|\Big\}\\
&\le
\sup_{\tau, \text{bounded}} \ee^x \Big\{ \sum_{i=0}^\infty \ind{\tau \ge T + i} \big( \ind{A_{T+i}} e^{-\alpha_{T+i}} + \ind{A^c_{T+i}} e^{-\lambda(T+i)} \big)\|f\|\\
&\hspace{70pt}+ \ind{A_T} e^{-\alpha_T} 2\|g\| + \ind{A^c_T} e^{-\lambda T} 2\|g\| \Big\}\\
&\le
\|f\| \sum_{i=0}^\infty \prob^x\{A_{T+i}\} + \|f\| e^{-\lambda T} \frac{1}{1 - e^{-\lambda}} + \prob^x \{ A_T \} 2 \|g\| + e^{-\lambda T} 2\|g\|\\
&\le
\|f\| \Big( \frac{C}{p} e^{-\rho T} + e^{-\lambda T} \frac{1}{1 - e^{-\lambda}} \Big) + 2 \|g\| \Big( C e^{-\rho T} + e^{-\lambda T} \Big).
\end{align*}
Hence, value functions $w_T$ converge to $w$ uniformly (and exponentially fast) on compact sets. Since $w_T$ are continuous, so is $w$.
\end{proof}
It is now standard to show that $\tau_\ve = \inf \{ t \ge 0:\ g(X_t) + \ve \ge w(X_t) \}$ is an $\ve$-optimal stopping time for any $\ve > 0$.

\section{Examples}\label{sec:example}

\textbf{Example 1.} Consider a diffusion on $\er^d$
\begin{equation}\label{eqn:diffusion}
dX_t = b(X_t) dt + \sigma(X_t) dW_t,
\end{equation}
where $(W_t)$ is a standard $d_1$-dimensional Brownian motion with $d_1 \ge d$, $b:\er^d \to \er^d$ and $\sigma:\er^d \to \er^{d \times d_1}$ are bounded Borel measurable functions. The boundedness of $b$ outside of a large enough ball is actually implied by the conditions spelled out below, so it is enough to assume local boundedness at this stage. Let $a(x) = \sigma(x) \sigma^T(x)$ and assume \emph{uniform non-degeneracy}, i.e., $\inf_{\|\xi\| = 1} \inf_x \xi^T a(x) \xi > 0$. The speed of convergence of $X_t$ to its invariant measure will depend on the quantity
\[
\eta(x) = \frac{x^T}{\|x\|} b(x) .
\]
The following list summarises three cases we will describe in detail. Each set of conditions implies assumptions \ref{ass:speed}, \ref{ass:integrability} and \ref{ass:uniform_int}. It is easy to verify that the Ornstein-Uhlenbeck process belongs to the first class.
\begin{itemize}
\item If $\eta(x) \le - r$ for sufficiently large $x$ and some $r > 0$, then the process converges at an exponential speed.
\item If $\eta(x) \le - r / \|x\|^p$ for sufficiently large $x$ and some $r > 0$, $0 < p < 1$, then the process converges at a subexponential speed.
\item If $\eta(x) \le - r / \|x\|$ for sufficiently large $x$ and some $r > d/2 + 1$, then the process converges at a polynomial speed. In particular, it suffices to show that $\limsup_{\|x\| \to \infty} x^T b(x) = - \infty$.
\end{itemize}
Details are collected in two lemmas below (Lemma \ref{lem:ver_subexp} and Lemma \ref{lem:ver_poly}). We shall omit references for the first class as it is contained in the second one, which provides sufficient conditions for this paper. Although the same could be said about the relationship between the second and third class, we present details for both cases as they are less well known.

\begin{lemma}[\cite{Klokov2005}] \label{lem:ver_subexp}
Assume that there is $M > 0$, $r > 0$ and $0 < p < 1$ such that $\eta(x) \le - r / \|x\|^p$ for $\|x\|> M$. Then there are a unique invariant measure $\mu$ and constants $C_0, C_1, A_0, A_1$ such that
\[
\sup_{t \ge 0} \ee^x \{ e^{2A_0 \|X_t\|^\alpha} \} \le C_0 e^{2A_0 \|x\|^\alpha}
\]
and
\[
\|P_t(x, \cdot)  - \mu(\cdot) \|_{TV} \le C_1 e^{A_0 \|x\|^\alpha - A_1 t^\delta},
\]
where $\alpha = 1-p$ and $\delta = (1-p)/(1+p)$.
\end{lemma}

\begin{lemma}[\cite{Veretennikov1997}]\label{lem:ver_poly}
Assume that there is $M > 0$ and $r > d/2 + 1$ such that $\eta(x) \le - r / \|x\|$ for $\|x\|> M$. Then there are a unique invariant measure $\mu$ and constants $m, k, C > 0,$ such that
\[
\|P_t(x, \cdot)  - \mu(\cdot) \|_{TV} \le C (1 + \|x\|^m) (1 + t)^{-(1+k)}
\]
and
\[
\sup_{t \ge 0} \ee^x \{ \|X_t\|^{m+\delta} \} \le C_0 (1 + \|x\|^{m+\delta})
\]
for sufficiently small $\delta > 0$.
\end{lemma}

\textbf{Example 2.} The reader is referred to \cite{Masuda2007} for details. Consider a jump-diffusion
\[
dX_t = b(X_t) dt + \sigma(X_t) dW_t + \int_{|z| \le 1} \zeta (X_{t-}, z) \tl \mu (dt, dz) + \int_{|z| > 1} \zeta (X_{t-}, z) \mu (dt, dz),
\]
where $b: \er^d \to \er^d$, $\sigma: \er^d \to \er^{d \times d_1}$ and $\zeta: \er^d \times \er^{d_2} \to \er^d$, $(W_t)$ is a $d_1$-dimensional standard Brownian motion and $\mu$ is a Poisson measure on $[0, \infty) \times (\er^{d_2} \setminus \{0\})$ with the intensity $\nu(dz) dt$. The Levy measure $\nu$ satisfies $\int (|z| \wedge 1) \nu(dz) < \infty$. $\tl \mu$ is the compensated measure $\mu$, i.e., $\tl \mu(dt, dz) = \mu(dt, dz) - \nu(dz) dt$. Let Assumptions 1, 2 and 3$^*$ from \cite{Masuda2007} be satisfied. The first one enforces linear growth for $b, \sigma, \zeta$, the second one is concerned with absolute continuity of the the transition probability with respect to the Lebesgue measure. The last assumption is a Foster-Lyapunov criterion. Theorem 2.2 in the aforementioned paper shows the existence of a unique invariant measure. Coupled with Proposition 3.8 therein it proves our assumptions \ref{ass:speed}-\ref{ass:integrability}. As a special case consider an Ornstein-Uhlenbeck process
\[
dX_t = -Q X_t dt + dZ_t,
\]
where $(Z_t)$ is a (multi-dimensional) Levy process with the generating triplet $(b, A, \nu)$ and all eigenvalues of $Q$ have positive real parts. If $\int_{|z|>1} |z|^q \nu(dz) < \infty$ for some $q > 0$ then the process $(X_t)$ satisfies  \ref{ass:speed}-\ref{ass:integrability} \cite[Theorem 2.6]{Masuda2007}; in fact, the convergence is exponentially fast. It can also be shown that it is strongly Feller if either (a) $A$ is of full rank or (b) $\nu(\er^d) = \infty$ and $\nu$ is absolutely continuous with respect to the Lebesgue measure \cite[Theorem 3.1]{Masuda2004}.



\bibliographystyle{amsplain}
\bibliography{references-undis}
\end{document}